\documentclass[12pt]{article}

\usepackage[utf8]{inputenc}
\usepackage[T1]{fontenc}
\usepackage{mathtools}
\usepackage{amsmath,diagbox}
\usepackage{indentfirst}
\usepackage{mathrsfs}
\usepackage{amsfonts}
\usepackage{arydshln}
\usepackage{enumerate}
\usepackage{setspace}
\usepackage{amssymb,amsthm,cases}
\usepackage{amsbsy,pifont}
\usepackage{latexsym,diagbox,tabularx}
\usepackage{amsmath,amscd,bbm,multirow}
\usepackage{graphicx,epsfig,extarrows,dsfont,mathtools}
\usepackage[final,allcolors=blue,colorlinks=true]{hyperref}
\usepackage[normalem]{ulem}
\input xy
\usepackage{caption,subcaption} 
\usepackage{tikz}
\usepackage{tikz-cd}
\usetikzlibrary{decorations.pathreplacing,decorations.markings}
\usetikzlibrary{cd}
\usetikzlibrary{shapes.geometric}
\usetikzlibrary{arrows,arrows.meta}
\usetikzlibrary{graphs,positioning}
\usetikzlibrary{matrix,decorations.pathmorphing}
\usetikzlibrary{math,calc,intersections,through,angles,arrows.meta,shapes.geometric,shadows,quotes,spy,datavisualization,datavisualization.formats.functions,plotmarks}
\tikzset{every picture/.style={samples=300,smooth,line join=round,thick,>=stealth}}
\usepackage{enumerate}
\usepackage{fancyhdr}
\usepackage[sectionbib]{chapterbib}

\tikzset{
	on each segment/.style={
		decorate,
		decoration={
			show path construction,
			moveto code={},
			lineto code={
				\path [#1]
				(\tikzinputsegmentfirst) -- (\tikzinputsegmentlast);
			},
			curveto code={
				\path [#1] (\tikzinputsegmentfirst)
				.. controls
				(\tikzinputsegmentsupporta) and (\tikzinputsegmentsupportb)
				..
				(\tikzinputsegmentlast);
			},
			closepath code={
				\path [#1]
				(\tikzinputsegmentfirst) -- (\tikzinputsegmentlast);
			},
		},
	},
	mid arrow/.style={postaction={decorate,decoration={
				markings,
				mark=at position .5 with {\arrow[#1]{stealth}}
	}}},
}

\numberwithin{equation}{section}

\theoremstyle{plain}
\newtheorem{theorem}{Theorem}[section]

\newtheorem{corollary}[theorem]{Corollary}

\newtheorem{lemma}[theorem]{Lemma}
\newtheorem{proposition}[theorem]{Proposition}

\theoremstyle{definition}
\newtheorem{definition}[theorem]{Definition}

\newtheorem{remark}[theorem]{Remark}

\newtheorem{assumption}[theorem]{Assumption}
\newtheorem{notation}[theorem]{Notation}

\def\XXint#1#2#3{{\setbox0=\hbox{$#1{#2#3}{\int}$}
		\vcenter{\hbox{$#2#3$}}\kern-.5\wd0}}

\DeclareMathSymbol{\subseteq}{\mathrel}{symbols}{"12}
\DeclareMathSymbol{\supseteq}{\mathrel}{symbols}{"13} 
\DeclareMathSymbol{\subsetneq}{\mathrel}{AMSb}{"28}                  \DeclareMathSymbol{\supsetneq}{\mathrel}{AMSb}{"29}    

\DeclareMathSymbol{\nsubseteq}{\mathrel}{AMSb}{"2A}                  

\DeclareMathSymbol{\nsupseteq}{\mathrel}{AMSb}{"2B}

\DeclareMathDelimiter{\langle}{\mathop}{symbols}{"68}{largesymbols}{"0A}
\DeclareMathDelimiter{\rangle}{\mathclose}{symbols}{"69}{largesymbols}{"0B}

\DeclareSymbolFont{txfontsA}{U}{txmia}{m}{it}
\SetSymbolFont{txfontsA}{bold}{U}{txmia}{bx}{it}
\DeclareFontSubstitution{U}{txmia}{m}{it}
\DeclareMathSymbol{\upalpha}{\mathord}{txfontsA}{"0B}
\DeclareMathSymbol{\upbeta}{\mathord}{txfontsA}{"0C}
\DeclareMathSymbol{\upgamma}{\mathord}{txfontsA}{"0D}
\DeclareMathSymbol{\updelta}{\mathord}{txfontsA}{"0E}
\DeclareMathSymbol{\upepsilon}{\mathord}{txfontsA}{"0F}
\DeclareMathSymbol{\upzeta}{\mathord}{txfontsA}{"10}
\DeclareMathSymbol{\upeta}{\mathord}{txfontsA}{"11}
\DeclareMathSymbol{\uptheta}{\mathord}{txfontsA}{"12}
\DeclareMathSymbol{\upiota}{\mathord}{txfontsA}{"13}
\DeclareMathSymbol{\upkappa}{\mathord}{txfontsA}{"14}
\DeclareMathSymbol{\uplambda}{\mathord}{txfontsA}{"15}
\DeclareMathSymbol{\upmu}{\mathord}{txfontsA}{"16}
\DeclareMathSymbol{\upnu}{\mathord}{txfontsA}{"17}
\DeclareMathSymbol{\upxi}{\mathord}{txfontsA}{"18}
\DeclareMathSymbol{\uppi}{\mathord}{txfontsA}{"19}
\DeclareMathSymbol{\uprho}{\mathord}{txfontsA}{"1A}
\DeclareMathSymbol{\upsigma}{\mathord}{txfontsA}{"1B}
\DeclareMathSymbol{\uptau}{\mathord}{txfontsA}{"1C}
\DeclareMathSymbol{\upupsilon}{\mathord}{txfontsA}{"1D}
\DeclareMathSymbol{\upphi}{\mathord}{txfontsA}{"1E}
\DeclareMathSymbol{\upchi}{\mathord}{txfontsA}{"1F}
\DeclareMathSymbol{\uppsi}{\mathord}{txfontsA}{"20}
\DeclareMathSymbol{\upomega}{\mathord}{txfontsA}{"21}
\DeclareMathSymbol{\upvarepsilon}{\mathord}{txfontsA}{"22}
\DeclareMathSymbol{\upvartheta}{\mathord}{txfontsA}{"23}
\DeclareMathSymbol{\upvarpi}{\mathord}{txfontsA}{"24}
\DeclareMathSymbol{\upvarrho}{\mathord}{txfontsA}{"25}
\DeclareMathSymbol{\upvarsigma}{\mathord}{txfontsA}{"26}
\DeclareMathSymbol{\upvarphi}{\mathord}{txfontsA}{"27}                   

\DeclareSymbolFont{ugmL}{OMX}{mdugm}{m}{n}
\SetSymbolFont{ugmL}{bold}{OMX}{mdugm}{b}{n}
\DeclareMathAccent{\wideparen}{\mathord}{ugmL}{"F3}

\def\pt{\partial}

\def\ra{\rightarrow}
\def\bs{\boldsymbol}

\def\s{\subseteq}

\def\e{\varepsilon}

\def\ol{\overline}

\def\vp{\varphi}
\def\lg{\langle}
\def\llg{\left\langle}

\def\rg{\rangle}
\def\rrg{\right\rangle}
\def\es{\emptyset}

\def\pt{\partial}

\def\Om{\Omega}
\def\la{\lambda}
\def\al{\alpha}

\def\de{\delta}
\def\ga{\gamma}

\def\Ga{\Gamma}

\def\ts{\times}

\def\iy{\infty}

\def\f{\frac}

\def\Lra{\Leftrightarrow}

\def\Span{{\rm{span}}}

\def\df{\mathrm d}

\def\wt{\widetilde}
\def\wh{\widehat}

\def\esssup{\operatorname*{ess\ \! sup}}

\def\hra{\hookrightarrow}

\def\mcA{\mathcal{A}}

\def\mcM{\mathcal{M}}

	\DeclareMathOperator{\Div}{div}

	\DeclareMathOperator{\dist}{dist}

	\DeclareMathOperator{\supp}{{supp}}

	\newcommand{\R}{\mathbb R}
	\newcommand{\N}{\mathbb N}

	\newcommand{\pxi}{\f{\pt}{\pt x^i}}
	\newcommand{\pxj}{\f{\pt}{\pt x^j}}
	\newcommand{\pxk}{\f{\pt}{\pt x^k}}

\begin{document} 
	
	
	\title{Shape-Design Approximation for a Class of Degenerate Hyperbolic Equations with a Degenerate Boundary Point and Its Application to Observability}
	\author{Dong-Hui Yang \and Jie Zhong}

	\maketitle{}
	\thispagestyle{empty}

	\begin{abstract}
	We study a class of degenerate hyperbolic equations in a bounded domain whose degeneracy occurs at a boundary point. We first develop the weighted functional framework, prove well-posedness of the degenerate problem, and establish regularity away from the degenerate point. We then introduce a shape-design approximation obtained by removing a small neighborhood of the degenerate boundary point, which yields uniformly non-degenerate hyperbolic problems on regularized domains. We prove that the regularized solutions converge to the solution of the original degenerate equation, including the convergence of the boundary normal derivatives away from the degenerate point. Finally, under a geometric condition on the observation boundary, we derive an observability inequality for the degenerate equation by combining the uniform observability of the regularized problems with the limit passage.
	\end{abstract}
	
	\section{Introduction}\label{S1}
	
	Shape design and domain approximation provide a useful way to study partial differential equations with singular coefficients or singular geometry \cite{Buttazzo,Chenais,Greco,Guo2,Guo1,Guo,He,Privat,Yang2}. For degenerate equations, this idea is particularly natural because the original model may exhibit weak traces, nonstandard boundary behavior, or limited regularity near the degenerate set. Replacing the degenerate problem by a family of uniformly non-degenerate problems on regularized domains offers a practical route from classical hyperbolic theory back to the degenerate equation.
	
	Controllability and observability for hyperbolic equations have been studied extensively \cite{Bai1,Cannarsa,Cannarsa1,Gueye,Lasiecka,Lasiecka3,Lions,Yang3,Yao,Zhang,Zuazua}. By the standard duality principle, observability of the adjoint system is equivalent to controllability of the original system \cite{Lions}. For degenerate hyperbolic equations, the one-dimensional theory is by now much better understood, while the higher-dimensional theory remains far more limited. In particular, when the degeneracy occurs at a boundary point, both the correct functional setting and the treatment of boundary terms require additional care.

	One effective route for degenerate equations is to approximate the degenerate operator by a family of uniformly elliptic or hyperbolic problems. When the degeneracy occurs at an interior point, this approximation strategy has been used in several related works \cite{Wu,Wu1,Yang1,Yang2}. When the degeneracy occurs on the boundary, the corresponding regularization naturally leads to a shape-design approximation \cite{Guo}.
	
	The present paper is aimed mainly at the higher-dimensional boundary-degenerate setting. The one-dimensional case can be treated by methods already available in the literature, for instance \cite{Gueye}; the novelty here lies in combining the boundary-point geometry with a shape-design approximation and a subsequent observability analysis on the regularized domains.
	
	A central point of the paper is that the shape-design method is not introduced merely for convenience. Even under a favorable geometric sign condition near the degenerate point, it is difficult to apply the classical multiplier method directly to the original degenerate equation, because the required boundary integrations near the degenerate point are not immediately justified at the level of the natural weighted energy space. In particular, one must first control the boundary trace and the normal derivative away from the degenerate point and justify the relevant integration-by-parts identities. The shape-design approximation resolves this difficulty by transferring the argument to smooth regularized domains $\Omega_\varepsilon$, where the approximate equations are uniformly hyperbolic and the classical multiplier method applies in a standard way.
	
	In this paper, we consider a class of degenerate hyperbolic equations whose degeneracy is located at a boundary point. More precisely, we study
	\begin{equation}\label{03.06.m}
		\begin{cases}
			\partial_{tt}y-\Div(|x|^\alpha \nabla y)=f, &\text{in }Q, \\
			y=0, &\text{on }  \pt Q, \\
			y(0)=y^0, \pt_ty(0)=y^1, &\text{in }\Omega,
		\end{cases}
	\end{equation}
	where $\alpha\in (0,1)$ is a given constant, $\Omega\subset\mathbb{R}^N$ with $N\ge 2$ is a bounded domain with $0\in\pt\Om$ and $\partial\Omega\in C^2$, $Q=\Omega\times (0,T)$ with $T>0$ being a constant, 
	and $\pt Q=\pt\Om \ts (0,T)$, 
		$y^0\in H_0^1(\Omega;w)$ and $y^1\in L^2(\Om)$ are the initial data, and $f\in L^2(Q)$. Here, the weight Sobolev space $H_0^1(\Om;w)$ will be defined in Section \ref{S2}. 
	
	The paper has three main objectives: to build a weighted framework for \eqref{03.06.m} and prove well-posedness; to approximate \eqref{03.06.m} by uniformly hyperbolic equations posed on regularized domains obtained by removing a small neighborhood of the degenerate boundary point; and to use this approximation procedure to derive an observability estimate under a geometric condition on the observation boundary.
	
	The geometric assumption used in the paper is the following.

	\begin{assumption}\label{Assumption (H)}
		In this paper, we assume there exists $R_0>0$ such that 
		\begin{equation*}
			x\cdot \nu(x)\leq 0 \mbox{ for all }x\in \pt\Om\cap B(0,R_0).
		\end{equation*}
		Here and in what follows, we denote $B(x_0,r)=\{x\in \R^N\colon |x-x_0|<r\}$ for $x_0\in\R^N$. Hence, for each $\e\in (0,\f{1}{8}R_0)$, there exists a $C^2$ bounded domain $\Om_\e$, such that 
		\begin{equation*}
			\Om_\e\s \Om, \mbox{ and  }  B(0,\e)\cap \Om_\e=\es \mbox{ and  } \Om-\Om_\e\s B(0,2\e)
		\end{equation*}and 
		\begin{equation*}
			x\cdot \nu(x)\leq 0 \mbox{ for all } x\in \pt\Om_\e \cap B(0,R_0). 
		\end{equation*}
		Moreover, we assume $M=\sup_{x\in\Om}|x|+1$. Denote 
		\begin{equation*}
			\Ga_0=\left\{x\in\pt\Om\colon x\cdot \nu(x)>0 \right\},
		\end{equation*}
		then $\Ga_0\s \pt\Om- B(0,R_0)$. 
		
		Denote 
		\begin{equation*}
			w=|x|^\al, \quad \mcA \vp=-\Div(w\nabla \vp), 
		\end{equation*} 
		and 
		\begin{equation*}
			w_\e=w|_{\Om_\e}, \quad \mcA_\e\vp=-\Div(w_\e\nabla\vp) \mbox{ on }\Om_\e. 
		\end{equation*}
	\end{assumption} 
	
	\begin{remark}\label{03.03.R1}
		This remark illustrates which local boundary geometries near the degenerate point satisfy Assumption \ref{Assumption (H)}. Let $N=2$ and $\e\in (0,\f{1}{8}R_0)$. Suppose that, in a neighborhood of the point $0\in \pt\Om$, the boundary is represented by a local chart $x_2=\ga(x_1)$. If $\ga(x_1)=-x_1^2$, so that locally
		\begin{equation*}
			\Om\cap B(0,\e)=B(0,\e)\cap \{x=(x_1,x_2)\in\R^2\colon x_2>-(x_1)^2 \},
		\end{equation*}
		then Assumption \ref{Assumption (H)} is satisfied. In contrast, if $\ga(x_1)=x_1^2$, namely
		\begin{equation*}
			\Om\cap B(0,\e)=B(0,\e)\cap \left\{x=(x_1,x_2)\in\R^2\colon x_2>(x_1)^2\right\},
		\end{equation*}
		then Assumption \ref{Assumption (H)} fails. The same is true for oscillatory profiles such as $\ga(x_1)=x_1^3\sin \f{1}{x_1}$, i.e.,
		\begin{equation*}
			\Om\cap B(0,\e)=B(0,\e)\cap \left\{x=(x_1,x_2)\in\R^2\colon x_2>x_1^3\sin \f{1}{x_1}\right\}.
		\end{equation*}
	\end{remark}
	
	This assumption has two roles. First, it makes it possible to construct a family of regularized domains $\Omega_\varepsilon$ that avoid the degenerate boundary point while preserving the relevant part of the outer boundary. Second, it yields the sign condition needed in the multiplier argument for the approximate observability estimate. Geometrically, the condition means that near the degenerate point the outward normal does not point in the radial direction of the multiplier field $H(x)=x$, so the corresponding boundary contribution is favorable. The same sign structure is inherited by the artificial boundary generated in the regularization process.
	
	The main results can be summarized as follows. Theorem \ref{12.12.T2} proves that the regularized solutions converge to the solution of the degenerate equation, both in the energy space and, away from the degenerate point, at the level of boundary normal derivatives. Theorem \ref{12.04.T4} states the resulting observability property for the degenerate equation. The role of the shape-design approximation is therefore twofold: first, it produces a family of uniformly non-degenerate problems on $\Omega_\varepsilon$, where the observability estimate can be established by classical arguments; second, it provides the convergence mechanism needed to pass that estimate back to the original degenerate equation. In this sense, the approximation is the bridge between the regularized observability analysis and the final observability result for the degenerate problem.
	
	\begin{theorem}\label{12.12.T2}
	Let $\al\in (0,1)$ and $\e\in (0,\f{1}{8}R_0)$. Let $y^0\in C_0^\iy(\Om)$, $y^1\in L^2(\Om)$, and $f\in L^2(Q)$. Suppose $y$ is the solution of \eqref{03.06.m} with respect to $(y^0,y^1,f)$, and $y_\e$ is the solution of \eqref{12.12.15} with respect to $(y_\e^0=y^0, y_\e^1=y^1|_{\Om_\e},f_\e=f|_{Q_\e})$. Then there exists a subsequence, still denoted by $\e$, such that
		\begin{equation}\label{12.12.20}
			\begin{split}
				Ey_\e
				&\ra y\ \!\quad  \mbox{ weakly in } L^2(0,T; H_0^1(\Om;w)), \\
				\pt_t Ey_\e 
				&\ra \pt_ty\ \mbox{ weakly in } L^2(Q),\\
				\f{\pt y_\e}{\pt \nu}=\f{\pt Ey_\e}{\pt \nu}
				&\ra \f{\pt y}{\pt \nu}\ \mbox{ strongly in } L^2(0,T; L^2(\pt\Om-B(0,R_0))).
			\end{split}
		\end{equation}
		Moreover, if we additionally assume $y^0\in C_0^\iy(\Om)$, $y^1\in C_0^\iy(\Om)$, and $f=0$, then there exists a subsequence, still denoted by $\e$, such that
		\begin{equation}\label{12.13.3}
			\f{\pt y_\e}{\pt \nu}=\f{\pt Ey_\e}{\pt \nu}\ra \f{\pt y}{\pt \nu} \mbox{ strongly in } L^2(0,T; L^2(\pt\Om-B(0,R_0))). 
		\end{equation}
		
	\end{theorem}
	
	\begin{theorem}\label{12.04.T4}
		Let $a,b$ be defined in \eqref{11.26.3} and \eqref{11.26.5}, respectively. The system \eqref{03.06.m} is observable for $T>\f{2b}{a}$. 
	\end{theorem}
	
	The rest of the paper is organized as follows. Section \ref{S2} introduces the weighted Sobolev spaces, the degenerate operator, and the well-posedness theory for \eqref{03.06.m}. Section \ref{S3} studies the shape-design approximation and proves convergence of the regularized solutions to the degenerate solution. Section \ref{S4} establishes observability for the approximate equations and then passes to the limit to obtain observability for the original degenerate equation.
	
	\section{Functional Setting and Well-posedness}\label{S2}
	
	In this section, we prepare the weighted framework for the degenerate equation \eqref{03.06.m}. We first introduce the relevant weighted Sobolev spaces and the degenerate operator, then collect the basic inequalities and spectral properties needed later, and finally prove well-posedness together with the regularity estimates that will be used in the approximation and observability arguments.

	\subsection{Solution spaces}

	We define
	\begin{equation*}
		H^{1}(\Omega; w) = \left\{u \in L^2(\Omega) \colon \int_\Om (\nabla u\cdot \nabla u)w\df x<+\iy\right\},
	\end{equation*}
	where $\nabla u=(\pt_{x_1}u, \cdots, \pt_{x_N}u)$. The inner product and the norm on $H^1(\Omega; w)$ are defined by 
	\begin{equation*}
		(u, v)_{H^1(\Omega; w)} =\int_\Om uv\df x+\int_\Om (\nabla u\cdot \nabla v)w\df x, \quad \|u\|_{H^1(\Om;w)}=(u,u)_{H^1(\Om;w)}^\f{1}{2}, 
	\end{equation*}
	respectively. 
	We define
	\begin{equation*}
		H_{0}^1(\Omega; w) = \text{the closure of } C_0^\iy(\Omega) \text{ in } H^1(\Omega; w).
	\end{equation*} 
	We denote by 
	\begin{equation*}
		H^{-1}(\Omega; w)=\mbox{the dual space of } H_{0}^1(\Omega; w) \mbox{ with pivot space } L^2(\Om),
	\end{equation*}
	which is a subspace of $(C_0^\iy(\Om))'$. 
	
	Now, we define 
	\begin{equation*}
		H^2(\Om;w)=\left\{u\in H^1(\Om;w)\colon \int_\Om (\mcA u)^2\df x<+\iy\right\}.
	\end{equation*}
	Its inner product   and norm are defined by 
	\begin{equation*}
		(u,v)_{H^2(\Om;w)}=(u,v)_{H^1(\Om;w)}+\int_\Om (\mcA u)(\mcA v)\df x, \quad \|u\|_{H^2(\Om;w)}=(u,u)_{H^2(\Om;w)}^\f{1}{2}, 
	\end{equation*}
	respectively. It should be emphasized that, in general, $H^2(\Omega; w)$ does not coincide with $H^2(\Omega)$.
	
	It is well known that 
	\begin{equation*}
		H_0^1(\Om;w), \quad H^1(\Om;w), \mbox{ and } H^2(\Om;w)
	\end{equation*} 
	are Hilbert spaces (see \cite{GC,Heinonen}).
	
	We denote 
	\begin{equation*}
		D(\mcA)=H^2(\Om;w)\cap H_0^1(\Om;w).
	\end{equation*}

	The following lemma is a version of Hardy’s inequality.
	
	\begin{lemma}\label{08.15.L1}
		Let $N \geq 2$ and $\alpha \in (0, 1)$. Then, for all $u \in H_0^1(\Omega; w)$, the following inequality holds:
		\begin{equation*}
			(N - 2 + \alpha) \left\||x|^{\frac{\alpha}{2} - 1} u\right\|_{L^2(\Omega)} \leq 2 \|\nabla u\|_{L^2(\Omega; w)}.
		\end{equation*} 
	\end{lemma}

\begin{remark}\label{08.15.R0}
		The estimate in Lemma \ref{08.15.L1} is stated in the form needed for the higher-dimensional framework developed in this paper. Since the focus of the present work is the higher-dimensional boundary-point degenerate setting, we organize the functional analysis in this regime. The one-dimensional case is not the novelty of the paper and can be handled separately by the methods already available in the existing one-dimensional literature.
	\end{remark}

\begin{proof}[Proof of Lemma \ref{08.15.L1}]
		By the density argument,  we only need to prove the case $u\in C_0^\iy(\Om)$. Let $z\in C_0^\iy(\Om)$. Note that 
		\begin{equation*}
			\begin{split}
				&2\int_{\Om_\e}|x|^{\al-2}z(x\cdot \nabla z)\df x\\
				&=\int_{\Om_\e} |x|^{\al-2}x\cdot \nabla z^2\df x=\int_{\Om_\e} \Div(|x|^{\al-2}z^2x)\df x-\int_{\Om_\e}z^2\Div(|x|^{\al-2}x)\df x\\
				&=\int_{\Om\cap \pt\Om_\e}|x|^{\al-2}z^2(x\cdot\nu)\df S-(N+\al-2)\int_{\Om_\e}z^2|x|^{\al-2}\df x, 
			\end{split}
		\end{equation*}
		then from Assumption 
		\ref{Assumption (H)} we get 
		\begin{equation*}
			\begin{split}
				(N+\al-2)\int_{\Om_\e}|x|^{\al-2}z^2\df x
				&\leq  (N+\al-2)\int_{\Om_\e}|x|^{\al-2}z^2\df x-\int_{\Om\cap \pt\Om_\e}|x|^{\al-2}z^2(x\cdot\nu)\df S\\
				&=2\int_{\Om_\e}|x|^{\al-2}z(x\cdot\nabla z)\df x\leq 2\int_{\Om_\e}(|x|^\f{\al-2}{2}|z|)(|x|^\f{\al}{2}|\nabla z|)\df x\\
				&\leq 2\left(\int_{\Om_\e}|x|^{\al-2}z^2\df x\right)^\f{1}{2}\left(\int_{\Om_\e}|x|^\al |\nabla z|^2\df x\right)^\f{1}{2}, 
			\end{split}
		\end{equation*}
		and hence
		\begin{equation*}
			\begin{split}
				(N+\al-2)\left(\int_{\Om_\e}|x|^{\al-2}z^2\df x\right)^\f{1}{2}\leq 2\left(\int_{\Om_\e}|x|^\al |\nabla z|^2\df x\right)^\f{1}{2}.
			\end{split}
		\end{equation*}
		Letting $\e\ra 0$, from $\Om_\e\ra \Om$ as $\e\ra 0$ by Assumption \ref{Assumption (H)}, we get the desired. 
	\end{proof}

	\begin{remark}\label{08.16.R1}
		From Lemma \ref{08.15.L1} and Assumption \ref{Assumption (H)},  we have
		\begin{equation}\label{08.16.8}
			 \int_\Om u^2\df x\leq \f{4M^{2-\al}}{(N-2+\al)^2}\int_\Om w\nabla u\cdot \nabla u\df x
		\end{equation}
		for all $u\in H_0^1(\Om;w)$. This inequality is  a Poincar\'{e} inequality.
		In particular, the norm
		\begin{equation}\label{08.19.2}
			\|u\|_{H_0^1(\Omega; w)} = \left(\int_\Omega (\nabla u \cdot \nabla u) w \, \mathrm{d}x\right)^{\frac{1}{2}}
		\end{equation}
		is an equivalent norm in $H_0^1(\Omega; w)$. Hereafter, we use \eqref{08.19.2} to define the norm of $H_0^1(\Omega; w)$.
	\end{remark}

	\begin{lemma}\label{08.15.L4}
		The embedding $H_0^1(\Omega; w) \hookrightarrow L^2(\Omega)$ is compact.
	\end{lemma}
	
	\begin{proof}
		Let $\{u_n\}_{n\in\N}\s H_0^1(\Om;w)$ be a bounded sequence, i.e., there exists a positive constant $C_0$ such that $\|u_n\|_{H^1(\Om;w)}\leq C_0$. Then there exists a subsequence of $\{u_n\}_{n\in\N}$, still denoted by itself, and $u_0\in H^1(\Om;w)$ such that $u_n\ra u_0$ weakly in $H^1(\Om;w)$. Without loss of generality, we assume $u_0=0$, for otherwise, we replace $u_n$ by $u_n-u_0$. 
		
		Let $\e\in (0,\f{1}{8}R_0)$. Then, from Assumption \ref{Assumption (H)},  we have
		\begin{equation}\label{03.03.1}
			\begin{split}
				\int_{\Om-\Om_\e} u_n^2\df x
				&\leq \int_{B(0,2\e)}u_n^2\df x=\int_{B(0,2\e)}|x|^{2-\al}|x|^{\al-2}u_n^2\df x\\
				&\leq (2\e)^{2-\al}\int_\Om u_n^2\df x\leq \e^{2-\al}2^{2-\al}C_0^2. 
			\end{split}
		\end{equation}
		Now, since 
		\begin{equation*}
			\e^{2-\al}\int_{\Om_\e} |\nabla u_n|^2\df x\leq \int_{\Om_\e} w_\e|\nabla u_n|^2\df x\leq \int_\Om w|\nabla u_n|^2\df x\leq C_0
		\end{equation*}
		by Assumption \ref{Assumption (H)}, 
			we get $\{u_n|_{\Om_\e}\}_{n\in\N}\s H^1(\Om_\e)$. Note that the embedding  $H^1(\Om_\e)\hra L^2(\Om_\e)$ is compact, then there exists a subsequence $\{u_{n_k}\}_{k\in\N}$ of $\{u_n\}_{n\in\N}$ such that $u_{n_k}\ra 0$ strongly in $L^2(\Om_\e)$. Moreover, there exists $k_\e\in\N$, such that $\|u_{n_{k_\e}}\|_{L^2(\Om_\e)}<\f{1}{2}\e$. Together with this and \eqref{03.03.1}, we get $u_{n_{k_\e}}\ra 0$ strongly in $L^2(\Om)$ as $\e\ra 0$.
			This completes the proof of this lemma. 
	\end{proof}
	
\subsection{Spectrum}

Consider the following degenerate elliptic equation
\begin{equation}\label{03.03.2}
	\begin{cases}
		\mcA u=f, &\mbox{in }\Om,\\
		u=0, &\mbox{on }\pt\Om, 
	\end{cases}
\end{equation}
where $f\in L^2(\Om)$. 
We say that $u\in H_0^1(\Om;w)$ is a weak solution to \eqref{03.03.2} with respect to $f$ if
\begin{equation}\label{03.03.3}
	\int_\Om (\nabla u\cdot \nabla v)w\df x=\int_\Om f v\df x
\end{equation}
for all $v\in H_0^1(\Om;w)$. 

\begin{lemma}\label{03.03.L1}
	Let $f\in L^2(\Om)$. Then there exists a unique weak solution $u\in D(\mcA)$ to \eqref{03.03.2}. Moreover, we have
	\begin{equation}\label{03.03.4}
		\int_\Om w|\nabla u|^2\df x\leq C\int_\Om f^2\df x, 
	\end{equation}
	where the positive constant $C$ depends only on $\al$ and $M$. 
\end{lemma}

\begin{proof}
	Set 
	\begin{equation*}
		B[u,v]=\int_\Om (\nabla u\cdot \nabla v)w\df x, \mbox{ for all } u,v\in H_0^1(\Om;w). 
	\end{equation*}
	Then $B[u,u]=\|u\|_{H_0^1(\Om;w)}$ and $|B[u,v]|\leq \|u\|_{H_0^1(\Om;w)}\|v\|_{H_0^1(\Om;w)}$. Since $f\in L^2(\Om)\s H^{-1}(\Om;w)$, the Lax--Milgram theorem yields a unique weak solution $u\in H_0^1(\Om;w)$ with respect to $f$ such that 
	\begin{equation*}
		B[u,v]=\int_\Om fv\df x, \mbox{ for all } v\in H_0^1(\Om;w). 
	\end{equation*}
	
	Finally, from $u\in H_0^1(\Om;w)$ and Remark \ref{08.16.R1}, we have
	\begin{equation*}
		\begin{split}
			\int_\Om w|\nabla u|^2\df x=\int_\Om fu\df x\leq \|f\|_{L^2(\Om)}\|u\|_{L^2(\Om)}\leq C\|f\|_{L^2(\Om)}\|u\|_{H_0^1(\Om;w)}, 
		\end{split}
	\end{equation*}
	this proves \eqref{03.03.4}. Moreover, \eqref{03.03.2} yields $\mcA u=f\in L^2(\Om)$. This completes the proof of the lemma. 
\end{proof}
	
From Remark \ref{08.16.R1} and Lemma \ref{08.15.L4} we obtain that the degenerate partial differential operator $\mcA$ has discrete point spectrum 
\begin{equation}\label{11.14.3}
	0<\la_1< \la_2\leq \la_3\leq  \cdots \ra +\iy, 
\end{equation}
i.e., $\la_n\ (n\in\N)$ satisfies the following equation
\begin{equation}\label{12.12.8}
	\begin{cases}
		\mcA \Phi_n=\la_n \Phi_n, & \mbox{in }\Om, \\
			\Phi_n=0, &\mbox{on }\pt\Om.
	\end{cases}
\end{equation}
Moreover, from \eqref{08.16.8},  we get
\begin{equation}\label{11.30.9}
	\la_1=\inf_{0\neq u\in H_0^1(\Om;w)}\f{\int_\Om w\nabla u\cdot \nabla u\df x}{\int_\Om u^2\df x}\geq \f{(N-2+\al)^2}{4M^{2-\al}}.
\end{equation}

\begin{notation} \label{12.11.N1}
We denote by $\Phi_n(x)$ the $n$th eigenfunction of $\mcA$ associated with the eigenvalue $\la_n$ for $n\in\N$. Then $\{\Phi_n\}_{n\in\N}$ is an orthonormal basis of $L^2(\Om)$ and an orthogonal family in $H_0^1(\Om;w)$.
(See \cite[Theorem 7 (pp. 728) in Appendix D]{Evans}, or the proof of Lemma \ref{06.28.L1} in the following.) 
\end{notation}

\begin{lemma}\label{06.28.L1}
	Let $u=\sum_{i=1}^\iy u_i \Phi_i\in H_0^1(\Om;w)$ with $u_i=(u,\Phi_i)_{L^2(\Om)}$ for all $i\in\N$. We have $\nabla u=\sum_{i=1}^\iy u_i\nabla \Phi_i$ and $\|u\|_{H_0^1(\Om;w)}=(\sum_{i=1}^\iy u_i^2\la_i)^\f{1}{2}$, and
	\begin{equation*}
		u\in H^2(\Om;w)\Lra \sum_{i=1}^\iy u_i^2\la_i^2<\iy,
	\end{equation*}
	and
	\begin{equation*}
		\mcA u=\sum_{i=1}^\iy u_i\la_i\Phi_i \mbox{ and } \|\mcA u\|_{L^2(\Om)}=\left(\sum_{i=1}^\iy u_i^2\la_i^2\right)^\f{1}{2}.
	\end{equation*}
\end{lemma}

\begin{proof}
	From \eqref{11.14.3}, we have
	\begin{equation}\label{06.04.3}
		\int_\Om (\nabla \Phi_k\cdot\nabla \Phi_l)w\df x=\de_{kl}\la_k \mbox{ for } k,l\in \N,
	\end{equation}
	where $\de_{kl}$ is the Kronecker delta function, defined as $\de_{kl}=1$ for $k=l$ and $\de_{kl}=0$ for $k\neq l$.
	We now demonstrate that $\{\la_k^{-\f{1}{2}}\Phi_k\}_{k=1}^\iy$ forms an orthonormal basis of $H_0^1(\Om;w)$.
	From \eqref{06.04.3}, it is straightforward to verify that $\{\la_k^{-\f{1}{2}}\Phi_k\}_{k=1}^\iy$ is an orthonormal subset of $H_0^1(\Om;w)$. To prove it is a basis, assume by contradiction that there exists $0\neq u\in H_0^1(\Om;w)$ such that
	\begin{equation*}
		\int_\Om (\nabla \Phi_k\cdot \nabla u)w\df x=0 \mbox{ for all }k\in\N.
	\end{equation*}
	Given that $\{\Phi_k\}_{k\in\N}$ is an orthonormal basis of $L^2(\Om)$, for $u\in H_0^1(\Om;w)$, we can express
	\begin{equation}\label{06.04.4}
		u=\sum_{k=1}^\iy d_k\Phi_k \mbox{ where }  d_k=(u, \Phi_k)_{L^2(\Om)}, k\in\N.
	\end{equation}
	Then  we have
	\begin{equation*}
		0=(u,\Phi_k)_{H_0^1(\Om;w)}=\int_\Om (\nabla \Phi_k\cdot \nabla u)w\df x=\la_k\int_\Om \Phi_ku\df x =\la_kd_k,
	\end{equation*}
	which implies $d_k=0$. This leads to $u=0$, a contradiction. 
	
	From above, since $u\in H_0^1(\Om;w)$,  we have 
	\begin{equation*}
		\nabla u=\sum_{i=1}^\iy e_i\la_i^{-\f{1}{2}}\nabla\Phi_k, \mbox{ and } \|u\|_{H_0^1(\Om;w)}^2=\sum_{i=1}^\iy e_i^2,
	\end{equation*}
	and 
	\begin{equation*}
		e_i=\int_\Om \left(\nabla u\cdot \la_i^{-\f{1}{2}}\nabla\Phi_i\right) w\df x=\la_i^{-\f{1}{2}}\int_\Om (\nabla u\cdot \nabla \Phi_i)w\df x=\la_i^\f{1}{2}\int_\Om u\Phi_i\df x=\la_i^\f{1}{2}u_i, 
	\end{equation*} 
	hence,  $\nabla u=\sum_{i=1}^\iy u_i\nabla\Phi_i$. Moreover, $\|u\|_{H_0^1(\Om;w)}=(\sum_{i=1}^\iy \la_iu_i^2)^\f{1}{2}$. 
	
	Let $u\in H^2(\Om;w)$. Take $\vp_n=\sum_{i=1}^n u_i\Phi_i\in H_0^1(\Om;w)\cap H^2(\Om;w)$ for each $n\in\N$, then from
	\begin{equation*}
		\begin{split}
			(\mcA u, \mcA\vp_n)_{L^2(\Om)}
			&=\sum_{i=1}^n u_i\la_i \int_\Om (\mcA u)\Phi_i\df x=\sum_{i=1}^n u_i\la_i\int_\Om u\mcA\Phi_i\df x\\
			&=\sum_{i=1}^n u_i\la_i^2\int_\Om u\Phi_i\df x=\sum_{i=1}^n u_i^2\la_i^2
		\end{split}
	\end{equation*}
	and $\|\mcA\vp_n\|_{L^2(\Om)}^2=\sum_{i=1}^n u_i^2\la_i^2$ we obtain
	\begin{equation*}
		\sum_{i=1}^nu_i^2\la_i^2\leq \|\mcA u\|_{L^2(\Om)}^2
	\end{equation*}
	for all $n\in\N$ by Cauchy inequality. This implies that $\sum_{i=1}^\iy u_i^2\la_i^2\leq \|\mcA u\|_{L^2(\Om)}^2<\iy$.

	Let $\sum_{i=1}^\iy u_i^2\la_i^2<\iy$. For each  $\vp\in C_0^\iy(\Om)$,  we have
	\begin{equation*}
		\begin{split}
			(\mcA u, \vp)_{L^2(\Om)}
			&=\int_\Om (\nabla u\cdot \nabla\vp)w\df x=\sum_{i=1}^\iy u_i \int_\Om (\nabla\Phi_i\cdot \nabla\vp)w\df x=\sum_{i=1}^\iy u_i\la_i\int_\Om \Phi_i\vp\df x.
		\end{split}
	\end{equation*}
	Note that
	\begin{equation*}
		\int_\Om \left(\sum_{i=1}^n u_i\la_i\Phi_i\right)^2\df x=\sum_{i=1}^n u_i^2\la_i^2\leq \sum_{i=1}^\iy u_i^2\la_i^2<\iy \mbox{ for all } n\in\N,
	\end{equation*}
	i.e., $\sum_{i=1}^\iy u_i\la_i\Phi_i\in L^2(\Om)$. Hence
	\begin{equation*}
		(\mcA u, \vp)_{L^2(\Om)}=\left(\sum_{i=1}^\iy u_i\la_i\Phi_i, \vp\right)_{L^2(\Om)}.
	\end{equation*}
	This implies that $\|\mcA u\|_{L^2(\Om)}\leq \sum_{i=1}^\iy u_i^2\la_i^2$ and $\mcA u=\sum_{i=1}^\iy u_i\la_i\Phi_i$.
\end{proof}

\subsection{Existence of weak solution} 

\begin{definition}\label{12.11.D1}
	We say that 
	\begin{equation*}
		y\in L^2(0,T; H_0^1(\Om;w))\cap H^1(0,T; L^2(\Om))\cap H^2(0,T; H^{-1}(\Om;w))
	\end{equation*}
	is a weak solution of \eqref{03.06.m} with respect to $(y^0,y^1,f)$ if
	
	(i) for every $v\in H_0^1(\Om;w)$ and a.e. $t\in (0,T)$ we have 
	\begin{equation*}
		\lg\pt_{tt}y, v\rg_{H^{-1}(\Om;w),H_0^1(\Om;w)}+\int_\Om w\nabla y\cdot \nabla v\df x=\int_\Om fv\df x, 
	\end{equation*}
	
	(ii) $y(0)=y^0$ and $\pt_ty(0)=y^1$. 
\end{definition}

The following lemma, which appears as \cite[Lemma 2.3 (p.~61)]{Bellassoued}, will be employed in the proof of Theorem \ref{12.11.T1}.

\begin{lemma}\label{11.30.L2}
	Let $(\mcM,g)$ be a $C^m$-Riemannian manifold with compact boundary $\pt \mcM$. Then there exists a $C^{m-1}$-vector field $\bs{n}$ such that 
	\begin{equation*}
		\bs{n}(x)=\nu(x), \ x\in\pt \mcM, \mbox{ and } |\bs{n}(x)|\leq 1, \ x\in \mcM, 
	\end{equation*}
	where $\nu$ is the unit outward normal vector to $\pt \mcM$. 
\end{lemma}

Now, we are in a position to establish the existence of weak solutions to \eqref{12.11.1}.

\begin{theorem}\label{12.11.T1}
	Let $y^0\in H_0^1(\Om;w)$ and $y^1\in L^2(\Om)$ and $f\in L^2(Q)$. Under Assumption \ref{Assumption (H)}, there exists a unique weak solution 
	\begin{equation*}
		y\in L^2(0,T; H_0^1(\Om;w))\cap H^1(0,T; L^2(\Om))\cap H^2(0,T; H^{-1}(\Om;w))
	\end{equation*} 
	to the equation \eqref{03.06.m} with respect to $(y^0,y^1,f)$, and  we have 
	\begin{equation}\label{12.11.1}
		\begin{split}
			&\esssup_{t\in (0,T)}\left(\|y(t)\|_{H_0^1(\Om;w)}+\|\pt_ty(t)\|_{L^2(\Om)}\right)+\|\pt_{tt}y\|_{L^2(0,T; H^{-1}(\Om;w))}\\
			&+\left\|\f{\pt y}{\pt \nu}\right\|_{L^2(0,T; L^2(\pt\Om-B(0,\f{1}{2}R_0)))}\\
			&\leq C\left(\|y^0\|_{H_0^1(\Om;w)}+\|y^1\|_{L^2(\Om)}+\|f\|_{L^2(Q)}\right),
		\end{split}
	\end{equation}
	where the positive constant $C$ depends only on $\al, R_0, M, T$, and $\Om-B(0,\f{1}{2}R_0)$. 
	
	Furthermore, if in addition $y^0\in D(\mcA)$ and $y^1\in H_0^1(\Om;w)$ and $f\in H^1(0,T; L^2(\Om))$, then 
	\begin{equation*}
		y\in L^2(0,T; D(\mcA))\cap H^1(0,T; H_0^1(\Om;w))\cap H^2(0,T; L^2(\Om)), 
	\end{equation*}
	and we have 
	\begin{equation}\label{12.11.2}
		\begin{split}
			&\esssup_{t\in (0,T)}\left(\|y(t)\|_{H^2(\Om-B(0,R_0))}+\|y(t)\|_{D(\mcA)}+\|\pt_ty(t)\|_{H_0^1(\Om;w)}+\|\pt_{tt}y(t)\|_{L^2(\Om)}\right)\\
			&+\left\|\f{\pt y}{\pt \nu}\right\|_{H^1(0,T; L^2(\pt\Om-B(0,\f{1}{2}R_0)))}\\
			&\leq C\left(\|y^0\|_{D(\mcA)}+\|y^1\|_{H_0^1(\Om;w)}+\|f\|_{H^1(0,T; L^2(\Om))}\right), 
		\end{split}
	\end{equation}
	where the positive constant $C$ depends only on $\al, R_0, M, T$, and $\Om-B(0,\f{1}{2}R_0)$. 
\end{theorem}

\begin{proof}
	We will prove this theorem by the following steps. 
	
	{\it Step 1}. We use the Galerkin method. 
	
	For each $m\in\N$, define 
	\begin{equation}\label{12.11.3}
		y^m(x,t)=\sum_{n=1}^m y_n^m(t)\Phi_n(x), \quad f^m(x,t)=\sum_{n=1}^m f_n^m(t)\Phi_n(x), 
	\end{equation}
	where $y_n^m(t)\ (n\leq m)$ satisfies the following equation 
	\begin{equation}\label{12.11.4}
		\pt_{tt}y_n^m(t)+\la_ny_n^m(t)=f_n^m(t), \ n=1,\cdots, m
	\end{equation}
	with initial data 
	\begin{equation}\label{12.11.6}
		y_n^m(0)=y_n^0,\quad  \pt_t y_n^m(0)=y_n^1, \ n=1,\cdots, m, 
	\end{equation}
	and $(n=1,\cdots, m)$
	\begin{equation}\label{12.11.5}
		f_n^m(t)=(f(t), \Phi_n)_{L^2(\Om)}, \quad y_n^0=(y^0,\Phi_n)_{L^2(\Om)},\quad y_n^1=(y^1,\Phi_n)_{L^2(\Om)}.  
	\end{equation}
	According to standard theory for ordinary differential equations, there exists $m$  functions $y_1^m, \cdots, y_m^m$ satisfying the equation \eqref{12.11.4} with initial data \eqref{12.11.6} and \eqref{12.11.5}. Note that \eqref{12.11.4} is equivalent to the following equality (using  Lemma \ref{06.28.L1})
	\begin{equation}\label{03.04.1}
		(\pt_{tt}y^m, \Phi_n)_{L^2(\Om)}+(y^m, \Phi_n)_{H_0^1(\Om;w)}=(f^m, \Phi_n)_{L^2(\Om)}, \ n=1,\cdots, m. 
	\end{equation}
	
	{\it Step 2}. Energy estimate. 
	
	Multiplying \eqref{03.04.1} by $\f{\df }{\df t}y_n^m(t)$, summing $k=1,\cdots, m$, then we have 
	\begin{equation}\label{12.12.1}
		\begin{split}
			\left(\pt_{tt} y^m, \pt_ty^m\right)_{L^2(\Om)}+(y^m,\pt_ty^m)_{H_0^1(\Om;w)}=(f^m, \pt_ty^m)_{L^2(\Om)}
		\end{split}
	\end{equation}
	for a.e. $t\in (0,T)$. Note that $(\pt_{tt}y^m, \pt_ty^m)_{L^2(\Om)}=\f{\df}{\df t}(\f{1}{2}\|\pt_ty^m\|_{L^2(\Om)}^2)$, and $(y^m,\pt_ty^m)_{H_0^1(\Om;w)}=\f{\df}{\df t}(\f{1}{2}\|y^m\|_{H_0^1(\Om;w)}^2)$, then we have 
	\begin{equation*}
		\begin{split}
			\f{\df}{\df t}\left(\|\pt_ty^m\|_{L^2(\Om)}^2+\|y^m\|_{H_0^1(\Om;w)}^2\right)\leq \|f^m(t)\|_{L^2(Q)}^2+\|\pt_ty^m\|_{L^2(\Om)}^2. 
		\end{split}
	\end{equation*}
	From Gronwall's inequality \cite[Appendix B.2 (j), p.~708]{Evans} we obtain 
	\begin{equation*}
		\begin{split}
			\|\pt_ty^m(t)\|_{L^2(\Om)}^2+\|y^m(t)\|_{H_0^1(\Om;w)}^2
			&\leq e^t\left(\|\pt_ty^m(0)\|_{L^2(\Om)}^2+\|y^m(0)\|_{H_0^1(\Om;w)}^2+\int_0^t\|f^m(s)\|_{L^2(\Om)}^2\df s\right) \\
			&\leq e^T\left(\|y^1\|_{L^2(\Om)}^2+\|y^0\|_{H_0^1(\Om;w)}^2+\|f\|_{L^2(Q)}^2\right)
		\end{split}
	\end{equation*}
	for each $t\in [0,T]$. 
	This implies that 
	\begin{equation}\label{12.12.2}
		\esssup_{t\in [0,T]}\left(\|\pt_ty^m\|_{L^2(\Om)}^2+\|y^m\|_{H_0^1(\Om;w)}^2\right)\leq e^T\left(\|y^0\|_{H_0^1(\Om;w)}^2+\|y^1\|_{L^2(\Om)}^2+\|f\|_{L^2(Q)}^2\right).  
	\end{equation} 
	
	{\it Step 3}. Energy estimate (continued). 
	
	Let $v\in H_0^1(\Om;w), \|v\|_{H_0^1(\Om;w)}\leq 1$, then we can write $v=v^1+v^2$, where $v^1\in \Span \{\Phi_n\}_{n=1}^m$ and $(v^2, \Phi_n)_{L^2(\Om)}=0$ for all $k=1,\cdots, m$. Note that from \eqref{12.11.4} we have 
	\begin{equation*}
		\lg \pt_{tt}y^m, v\rg_{H^{-1}(\Om;w),H_0^1(\Om;w)}=(\pt_{tt}y^m, v)_{L^2(\Om)}=(\pt_{tt}y^m, v^1)_{L^2(\Om)}=(f^m,v^1)_{L^2(\Om)}-(y^m,v^1)_{H_0^1(\Om;w)}, 
	\end{equation*}
	from H\"older inequality and  $\|v^1\|_{H_0^1(\Om;w)}\leq 1$ (see Lemma \ref{06.28.L1}) and \eqref{08.16.8} we obtain 
	\begin{equation*}
		\begin{split}
			\left|\lg \pt_{tt}y^m, v\rg_{H^{-1}(\Om;w),H_0^1(\Om;w)}\right|\leq C\left(\|f^m\|_{L^2(\Om)}+\|y^m\|_{H_0^1(\Om;w)}\right),
		\end{split}
	\end{equation*}
	where the constant $C>0$ depends only on $\al$ and $M$. This together with \eqref{12.12.2} we obtain 
	\begin{equation}\label{12.12.3}
		\begin{split}
			\int_0^T \|\pt_{tt}y^m\|_{H^{-1}(\Om;w)}^2\df t\leq C\left(\|y^0\|_{H_0^1(\Om;w)}^2+\|y^1\|_{L^2(\Om)}^2+\|f\|_{L^2(Q)}^2\right), 
		\end{split}
	\end{equation}
	where the constant $C>0$ depends only on $\al, T$ and $M$. 
	
	{\it Step 4}. Weak convergence. 
	
		From \eqref{12.12.2} and \eqref{12.12.3}, there exists a subsequence of $\{y^m\}_{m\in\N}$, still denoted by itself, and
	\begin{equation*}
		\wh y\in L^2(0,T; H_0^1(\Om;w))\cap H^1(0,T; L^2(\Om))\cap H^2(0,T;H^{-1}(\Om;w))
	\end{equation*}
	such that 
	\begin{equation}\label{12.12.4}
		\begin{split}
			y^m
			&\ra \wh y\ \ \  \ \mbox{ weakly star in } L^\iy(0,T; H_0^1(\Om;w)),\\
			\pt_ty^m
			&\ra \pt_t \wh y\  \ \!\mbox{  weakly star in } L^\iy(0,T; L^2(\Om)),\\
			\pt_{tt}y^m
			&\ra \pt_{tt}\wh y\ \mbox{ weakly in } L^2(0,T; H^{-1}(\Om;w)). 
		\end{split}
	\end{equation}
	Letting $m\ra\iy$, from \eqref{12.12.2} and \eqref{12.12.3},  we get
	\begin{equation}\label{12.12.12}
		\begin{split} 
		&\esssup_{t\in [0,T]}\left(\|\pt_t\wh y\|_{L^2(\Om)}^2+\|\wh y\|_{H_0^1(\Om;w)}^2\right)+\|\pt_{tt}\wh y\|_{L^2(0,T; H^{-1}(\Om;w))}\\
		&\leq C\left(\|y^0\|_{H_0^1(\Om;w)}^2+\|y^1\|_{L^2(\Om)}^2+\|f\|_{L^2(Q)}^2\right)  
		\end{split} 
	\end{equation} 
	for a.e. $t\in [0,T]$, where the constant $C>0$ depends only on $\al, T$ and $M$.  
	
	{\it Step 5}. We verify $\wh y$ satisfies Definition \ref{12.11.D1} (i). 
	
	Let $\psi(t)=\sum_{n=1}^k\psi_n(t)\Phi_n$, where $\{\psi_n(t)\}_{n=1}^k\s C^2[0,T]$. Multiplying \eqref{03.04.1} by $\psi_n(t)$, summing $n=1,\cdots, k$, integrating on  $[0,T]$, when $m\geq k$, we have 
	\begin{equation}\label{12.12.7}
		\int_0^T \lg \pt_{tt}y^m, v\rg_{H^{-1}(\Om;w),H_0^1(\Om;w)}\df t+\int_0^T (y^m, v)_{H_0^1(\Om;w)}\df t=\int_0^T (f^m,v)_{L^2(\Om)}\df t, 
	\end{equation}
	then, letting $m\ra \iy$, we obtain 
	\begin{equation}\label{12.12.5}
		\begin{split}
			\int_0^T\lg \pt_{tt}\wh y,v\rg_{H^{-1}(\Om;w),H_0^1(\Om;w)}\df t+\int_0^T (\wh y, v)_{H_0^1(\Om;w)}\df t=\int_0^T(f,v)_{L^2(\Om)}\df t. 
		\end{split}
	\end{equation}
	Note that 
	\begin{equation*}
		W=\left\{\sum_{n=1}^k\psi_n(t)\Phi_n\colon \psi_n(t) \mbox{ is a smooth function on } [0,T], n=1,\cdots, k, k\in\N\right\}
	\end{equation*}
	is dense in $L^2(0,T; H_0^1(\Om;w))$, Hence \eqref{12.12.5} hold for all $v\in L^2(0,T; H_0^1(\Om;w))$. Therefore, we have 
	\begin{equation*}
		\begin{split}
			\lg \pt_{tt}\wh y, v\rg_{H^{-1}(\Om;w),H_0^1(\Om;w)}+(\wh y,v)_{H_0^1(\Om;w)}=(f,v)_{L^2(\Om)}
		\end{split}
	\end{equation*}
	for all $v\in H_0^1(\Om;w)$ and a.e. $t\in [0,T]$. Moreover, 
	\begin{equation}\label{12.12.6}
		\wh y\in C([0,T]; L^2(\Om)), \mbox{ and } \pt_t\wh y \in C([0,T]; H^{-1}(\Om;w)). 
	\end{equation}
	
	{\it Step 6}. We prove $y$ satisfies Definition \ref{12.11.D1} (ii). 
	
	On the one hand, for each $\psi\in C^2([0,T]; H_0^1(\Om;w))$ with $\psi(T)=\pt_t\psi(T)=0$, from \eqref{12.12.5} and \eqref{12.12.6}, we obtain 
	\begin{equation*}
		\begin{split} 
		&\int_0^T(y, \pt_{tt}\psi)_{L^2(\Om)}\df t+\int_0^T (y,\psi)_{H_0^1(\Om;w)}\df t\\
		&=\int_0^T (f,\psi)_{L^2(\Om)}\df t-(y(0),\pt_t\psi(0))_{L^2(\Om)}+\lg\pt_ty(0), \psi(0)\rg_{H^{-1}(\Om;w),H_0^1(\Om;w)}. 
		\end{split} 
	\end{equation*} 
	On the other hand, from \eqref{12.12.7}, we obtain 
	\begin{equation*}
		\begin{split}
			&\int_0^T(y^m,\pt_{tt}\psi)_{L^2(\Om)}\df t+\int_0^T (y^m,\psi)_{L^2(\Om)}\df t\\
			&=\int_0^T (f^m,\psi)_{L^2(\Om)}\df t-(y^m(0),\pt_t\psi(0))_{L^2(\Om)}+(\pt_ty^m(0),\psi(0))_{L^2(\Om)}. 
		\end{split}
	\end{equation*}
	Letting $m\ra\iy$, from \eqref{12.12.4} and \eqref{12.11.6} and \eqref{12.11.5}, we obtain 
	\begin{equation*}
		\begin{split}
			&\int_0^T(y,\pt_{tt}\psi)_{L^2(\Om)}\df t+\int_0^T (y,\psi)_{L^2(\Om)}\df t\\
			&=\int_0^T (f,\psi)_{L^2(\Om)}\df t-(y^0,\pt_t\psi(0))_{L^2(\Om)}+(y^1,\psi(0))_{L^2(\Om)}. 
		\end{split}
	\end{equation*} 
	Overall, we obtain $y(0)=y^0, \pt_ty(0)=y^1$. 
	
	{\it Step 7}. The solution $y$ of equation \eqref{03.06.m} is unique. 
	
	We only need to show the solution of \eqref{03.06.m} with respect to $(0,0,0)$ is zero function. To verify this, fix $0\leq s\leq T$ and set
	\begin{equation*}
		z(t)=
		\begin{cases}
			\int_t^s y(\tau)\df \tau, &\mbox{if } 0\leq t\leq s, \\
			0, &\mbox{if } s\leq t\leq t. 
		\end{cases}
	\end{equation*}
	Then $z(t)\in H_0^1(\Om;w)$ for each $t\in [0,T]$, and hence
	\begin{equation*}
		\int_0^s \lg \pt_{tt} y, z\rg_{H^{-1}(\Om;w), H_0^1(\Om;w)}\df t+\int_0^s(y, z)_{H_0^1(\Om;w)}\df t=0. 
	\end{equation*}
	Since $\pt_ty(0)=z(s)=0$, integrating by parts, we get
	\begin{equation*}
		-\int_0^s (\pt_ty, \pt_tz)_{L^2(\Om)}\df t+\int_0^s (y,z)_{H_0^1(\Om;w)}\df t=0. 
	\end{equation*}
	Note that $\pt_tz=-y\ (0\leq t\leq s)$, and then 
	\begin{equation*}
		\int_0^s( \pt_ty,y)_{L^2(\Om)}\df t-\int_0^T(\pt_tz, z)_{L^2(\Om)}\df t=0. 
	\end{equation*}
	i.e., 
	\begin{equation*}
		\int_0^s\f{\df}{\df t}\left(\f{1}{2}\|y\|_{L^2(\Om)}^2-\f{1}{2}(z,z)_{H_0^1(\Om;w)}\right)\df t=0. 
	\end{equation*}
	Hence 
	\begin{equation*}
		\|y(s)\|_{L^2(\Om)}^2+\|z(0)\|_{H_0^1(\Om;w)}^2=0. 
	\end{equation*}
	This shows that $y=0$. 
	
	From this step and \eqref{12.12.12} we obtain the first three terms in \eqref{12.11.1}. 
	
	{\it Step 8}. Improve regularity. 
	
	We assume $y^0\in D(\mcA), y^1\in H_0^1(\Om;w)$ and $f\in H^1(0,T; L^2(\Om))$. In this case, we have
	\begin{equation*}
		\pt_{tt}\left(\pt_ty_n^m\right)(t)+\la_n \left(\pt_ty_n^m\right)(t)=(\pt_tf)_n^m(t), \ n=1,\cdots, m
	\end{equation*}
	is meaningful. 
	
	Denote $\wt y^m=\pt_t y^m$, from  \eqref{12.11.4} and \eqref{12.11.3} and \eqref{12.12.8}, we have 
	\begin{equation*}
		(\pt_{tt}\wt y^m, \Phi_n)_{L^2(\Om)}+(\wt y^m, \Phi_n)_{H_0^1(\Om;w)}=\left((\pt_t f)^m, \Phi_n\right)_{L^2(\Om)}, \ n=1,\cdots,m. 
	\end{equation*}
	Multiplying by $\pt_t\wt y_n^m=\pt_{tt}y_n^m(t)$, adding for $n=1,\cdots, m$, we get
	\begin{equation*}
		(\pt_{tt}\wt y^m, \pt_t\wt y^m)_{L^2(\Om)}+(\wt y^m, \pt_t \wt y^m)_{H_0^1(\Om;w)}=\left((\pt_tf)^m, \pt_t\wt y^m\right)_{L^2(\Om)}. 
	\end{equation*}
	Similar to Step 2, we get 
	\begin{equation}\label{12.12.9}
		\begin{split} 
		&\|\pt_t\wt y^m\|_{L^2(\Om)}^2+\|\wt y^m\|_{H_0^1(\Om;w)}\\
		&\leq e^t\left(\|\wt y^m(0)\|_{H_0^1(\Om;w)}^2+\|\pt_t\wt y^m(0)\|_{L^2(\Om)}^2+\int_0^t\|(\pt_tf)^m(s)\|_{L^2(\Om)}^2\df s\right)
		\end{split}  
	\end{equation}
	for all $t\in [0,T]$. 
	Here $(\pt_tf)^m=\sum_{n=1}^m(\pt_tf)_n^m\Phi_n=\sum_{n=1}^m (\pt_tf,\Phi_n)_{L^2(\Om)}\Phi_n$.
	
	Note that from \eqref{12.11.3} we have 
	\begin{equation*}
		\pt_t\wt y^m(0)=(\pt_{tt}y^m)(0)=\sum_{n=1}^m (\pt_{tt}y_n^m)(0)\Phi_n=\sum_{n=1}^m f_n(0)\Phi_n-\sum_{n=1}^m\la_ny_n^m(0)\Phi_n,
	\end{equation*} 
	and  
	\begin{equation*}  \sum_{n=1}^m |f_n(0)|^2 \leq \|f(0)\|_{L^2(\Om)}^2\leq C\left(\|f\|_{L^2(Q)}^2+\|\pt_tf\|_{L^2(Q)}^2\right)
	\end{equation*}
	by \cite[Section 5.9.2, Theorem 2(iii)]{Evans} and $f\in H^1(0,T; L^2(\Om))$, then 
	\begin{equation*}
		\begin{split} 
		\|\pt_t\wt y^m(0)\|_{L^2(\Om)}^2
		&\leq  2\sum_{n=1}^m |f_n(0)|^2+2\sum_{n=1}^m \la_n^2|y_n^m(0)|^2\\
		&\leq C\left(\|f\|_{H^1(0,T; L^2(\Om))}^2+\|\mcA y^0\|_{L^2(\Om)}^2\right)
		\end{split} 
	\end{equation*}
	by \eqref{12.11.5} and Lemma \ref{06.28.L1}, where the positive constant $C$ depends only on $T$. 
	Now, from  $\pt_ty^m(0)=\sum_{n=1}^m (\pt_ty_n^m)(0) \Phi_n$ and  \eqref{12.11.6} and \eqref{12.11.5} and Lemma \ref{06.28.L1}, we obtain
	\begin{equation*}
		\begin{split}
			\|\wt y^m\|_{H_0^1(\Om;w)}^2=\sum_{n=1}^m \left|\pt_ty_n^m(0)\right|^2\la_n\leq \|y^1\|_{H_0^1(\Om;w)}^2. 
		\end{split}
	\end{equation*}
	Combining these and \eqref{12.12.9} we get 
	\begin{equation}\label{09.04.9}
		\begin{split} 
			&\int_\Om |\pt_t\wt  y^m|^2\df x+\int_\Om (\nabla \wt y^m\cdot\nabla \wt y^m)w\df x\\ 
			&\leq C\left(\|\mcA y^0\|_{L^2(\Om)}^2+\|y^1\|_{H_0^1(\Om;w)}^2+\|f\|_{H^1(0,T; L^2(\Om))}^2\right)
		\end{split}
	\end{equation}
	for a.e. $t\in [0,T]$, where the constant $C>0$ depends only on $T$. 
	
	Consider 
	\begin{equation*}
		(y^m,\Phi_n)_{H_0^1(\Om)}=(f^m-\pt_{tt}y^m, \Phi_n)_{L^2(\Om)}, \ n=1,\cdots, m.
	\end{equation*} 
	Multiplying $\la_ny_n^m(t)$ on the both sides, summing $n=1,\cdots,m$, from $\mcA y^m\in H_0^1(\Om;w)$, we have
	\begin{equation*}
		(\mcA y^m, \mcA y^m)_{L^2(\Om)}=(y^m, \mcA y^m)_{H_0^1(\Om;w)}=(f^m-\pt_{tt}y^m, \mcA y^m)_{L^2(\Om)}.
	\end{equation*} 
	This implies that 
	\begin{equation}\label{12.12.10}
		\|\mcA y^m\|_{L^2(\Om)}\leq \|f^m\|_{L^2(\Om)}+\|\pt_{tt}y^m\|_{L^2(\Om)}.
	\end{equation} 
	
	Choosing $\zeta\in C^\iy(\R^N), 0\leq \zeta\leq 1$ such that 
	\begin{equation*}
		\begin{split}
			\zeta=1 \mbox{ on }\Om-B(0,R_0), \quad \zeta=0 \mbox{ on }B\left(0,\f{1}{2}R_0\right), \quad |\nabla \zeta|\leq CR_0^{-1}, \left|D^2\zeta\right|\leq CR_0^{-2}, 
		\end{split}
	\end{equation*}
	where the positive constants $C$ are absolute. Taking $\xi=\zeta y^m$, 
	since $\xi$ satisfies  the uniformly elliptic equation 
	\begin{equation*}
		\mcA \xi=g \mbox{ in }\Om-B\left(0,\f{1}{4}R_0\right),\quad \xi=0 \mbox{ in } \pt\left(\Om-B\left(0,\f{1}{4}R_0\right)\right)
	\end{equation*}
	with 
	\begin{equation*}
		g=-\zeta\pt_{tt}y^m+\zeta f^m-2w\nabla\zeta \cdot \nabla y^m-y^m \Div(|x|^\al \nabla\zeta), 
	\end{equation*}
	by the classical elliptic regularity theory \cite[Theorem 1 (p.~327) in Chapter 6.3.1 and Theorem 4 (p.~334) in Chapter 6.3.2]{Evans}, we have 
	\begin{equation*}
		\|\xi\|_{H^2(\Om-B(0,\f{1}{2}R_0))}\leq C\left(\|g\|_{L^2(\Om)}+\|\xi\|_{L^2(\Om)}\right), 
	\end{equation*} 
	where the constant $C>0$ depends only on $\al,R_0$ and $\Om-B(0,\f{1}{2}R_0)$. 
	Apply this to \eqref{09.04.9} we get 
	\begin{equation*}
		\|y^m\|_{H^2(\Om-B(0,\f{1}{2}R_0))}+\|\mcA y^m\|_{L^2(\Om)}\leq C\left(\|f^m\|_{L^2(\Om)}+\|\pt_{tt}y^m\|_{L^2(\Om)}+\|y^m\|_{H_0^1(\Om;w)}\right). 
	\end{equation*}
	Combining this with $\wt y^m=\pt_ty^m$, \eqref{12.12.2}, and \eqref{09.04.9}, we obtain 
	\begin{equation*} 
		\begin{split}
			&\esssup_{0\leq t\leq T}\left(\|y^m(t)\|_{H^2(\Om-B(0,\f{1}{2}R_0))}+\|y^m(t)\|_{D(\mcA)}+\|\pt_ty^m(t)\|_{H_0^1(\Om;w)}+\|\pt_{tt}y^m(t)\|_{L^2(\Om)}\right)\\
			&\leq C\left(\|y^0\|_{D(\mcA)}+\|y^1\|_{H_0^1(\Om;w)}+\|f \|_{H^1(0,T; L^2(\Om))}\right).
		\end{split}
	\end{equation*}
	Letting $m\ra\iy$, we obtain
	\begin{equation}\label{09.07.16}
		\begin{split}
			&\esssup_{0\leq t\leq T}\left(\|y(t)\|_{H^2(\Om-B(0,\f{1}{2}R_0))}+\|y(t)\|_{D(\mcA)}+\|\pt_ty(t)\|_{H_0^1(\Om;w)}+\|\pt_{tt}y(t)\|_{L^2(\Om)}\right)\\
			&\leq C\left(\|y^0\|_{D(\mcA)}+\|y^1\|_{H_0^1(\Om;w)}+\|f\|_{H^1(0,T; L^2(\Om))}\right), 
		\end{split}
	\end{equation} 
	where the constant $C>0$ depends only on $\al, R_0, T$, and $\Om-B(0,\f{1}{2}R_0)$. 
	
	From \eqref{12.12.12},  Steps 5–6, \eqref{09.07.16}, and \eqref{12.12.14}, we have proved the theorem except for the fifth term on the left-hand side of \eqref{12.11.2}.
	
	{\it Step 9}. Hidden regularity on $\pt\Om-B(0,R_0)$, i.e., we establish the fourth term in \eqref{12.11.1}. 
	
	Let $\bs{n}\in [C^2(\ol\Om)]^N$ be defined in Lemma \ref{11.30.L2}. Choosing $\zeta\in C_0^\iy(\R^N), 0\leq \zeta\leq 1$ such that 
	\begin{equation*}
		\zeta=1 \mbox{ on } \Om-B(0,R_0), \quad \zeta=0 \mbox{ in } B\left(0,\f{1}{2}R_0\right),\quad |\nabla \zeta|\leq CR_0^{-1} \mbox{ on }\R^N, 
	\end{equation*}
	where the constant $C>0$ is absolute. Multiplying $\zeta\bs{n}\cdot \nabla y$, integrating on $Q$, we have 
	\begin{equation}\label{12.12.11}
		\begin{split}
			\iint_Q (\pt_{tt}y)(\zeta\bs{n}\cdot \nabla y)\df x\df t-\iint_Q [\Div(w\nabla y)](\zeta \bs{n}\cdot \nabla y)\df x\df t=\iint_Q f(\zeta\bs{n}\cdot \nabla y)\df x\df t. 
		\end{split}
	\end{equation}

	Note that the first term in \eqref{09.07.16} and the support condition $\supp\zeta\s \R^N-B(0,\f{1}{4}R_0)$ ensure that the following integration by parts is justified. On the one hand, integrating by parts and using $\pt_ty=0$ on $\pt\Om$, we have 
	\begin{equation*}
		\begin{split}
			\iint_Q (\pt_{tt}y)(\zeta\bs{n}\cdot \nabla y)\df x\df t
			&=\int_\Om (\pt_ty)(\zeta\bs{n}\cdot \nabla y)\df x\bigg|_{t=0}^{t=T}+\f{1}{2}\iint_Q (\pt_ty)^2\Div(\zeta \bs{n})\df x\df t.
		\end{split}
	\end{equation*}
	From \eqref{12.12.12} we obtain 
	\begin{equation}\label{12.12.13}
		\begin{split}
			\left|\iint_Q (\pt_{tt}y)(\zeta \bs{n}\cdot \nabla y)\df x\df t\right|
			&\leq C\left(\|y^0\|_{H_0^1(\Om;w)}^2+\|y^1\|_{L^2(\Om)}^2+\|f\|_{L^2(Q)}^2\right), 
		\end{split}
	\end{equation}
	where the constant $C>0$ depends only on $\al, R_0, T$, and $\Om$. 
	On the other hand, integrating by parts and using $\bs{n}=(n_1,\cdots,n_N)=\nu$ on $\pt\Om$, $y=0$ on $\pt\Om$, $\supp \zeta \s \R^N-B(0,\f{1}{2}R_0)$, and $y\in L^2(0,T; H^2(\Om-B(0,\f{1}{2}R_0)))$, we get 
	\begin{equation*}
		\begin{split}
			&-\iint_Q [\Div(w\nabla y)](\zeta \bs{n}\cdot \nabla y)\df x\df t\\
			&=-\iint_Q \Div\left[(\zeta\bs{n}\cdot \nabla y)w\nabla y\right]\df x\df t+\iint_Q w\nabla y\cdot \nabla (\zeta\bs{n}\cdot \nabla y)\df x\df t\\
			&=-\iint_{\pt Q}\zeta w\left(\f{\pt y}{\pt \nu}\right)^2\df S\df t+\iint_Q w\nabla y\cdot (D(\zeta\bs{n})\nabla y)\df x\df t+\f{1}{2}\iint_Q w(\zeta\bs{n})\cdot \nabla \left|\nabla y\right|^2\df x\df t\\
			&=-\f{1}{2}\iint_{\pt Q} \zeta w\left(\f{\pt y}{\pt \nu}\right)^2\df S\df t+\iint_Q w\nabla y\cdot (D(\zeta\bs{n})\nabla y)\df x\df t-\f{1}{2}\iint_Q|\nabla y|^2\Div(\zeta w\bs{n})\df x\df t, 
		\end{split}
	\end{equation*}
	Combining this with \eqref{12.12.13} and \eqref{12.12.12}, we obtain the following estimate, where the constant depends only on $\al, R_0, T$, and $\Om-B(0,\f{1}{2}R_0)$:
	\begin{equation*}
		\begin{split}
			\left|\iint_Q f(\zeta\bs{n}\cdot \nabla y)\df x\df t\right|
			&\leq \f{1}{2}\iint_Q f^2\df x\df t+\f{1}{2}\iint_Q |\zeta\bs{n}\cdot\nabla y|\df x\df t\\
			&\leq C\left(\|y^0\|_{H_0^1(\Om;w)}^2+\|y^1\|_{L^2(\Om)}^2+\|f\|_{L^2(Q)}^2\right)
		\end{split}
	\end{equation*}
	we obtain 
	\begin{equation}\label{12.12.14}
		\begin{split}
			\int_0^T\int_{\Om-B(0,R_0)} w\left(\f{\pt y}{\pt \nu}\right)^2\df S\df t\leq C\left(\|y^0\|_{H_0^1(\Om;w)}^2+\|y^1\|_{L^2(\Om)}^2+\|f\|_{L^2(Q)}^2\right), 
		\end{split}
	\end{equation}
	where the constant $C>0$ depends only on $\al, T, R_0$, and $\Om-B(0,\f{1}{2}R_0)$. This proves the fourth term in \eqref{12.11.1}. 
	
	{\it Step 10}.  We show the fifth term of \eqref{12.11.2}. 
	
	Indeed, let $y^0\in D(\mcA)$ and $y^1\in H_0^1(\Om;w)$ and $f\in H^1(0,T; L^2(\Om))$. Then $h=\pt_ty$ is the solution of the following system
	\begin{equation*}
		\begin{cases}
			\pt_{tt}h+\mcA h=\pt_tf, &\mbox{in }Q, \\
			h=0, &\mbox{on }\pt Q, \\
			h(0)=y^1, \pt_th(0)=f(0)-\mcA y^0, &\mbox{in } \Om, 
		\end{cases}
	\end{equation*}
	this can be carried out as in Step 8. Finally, from this and \eqref{12.11.1}, we get the fifth term in \eqref{12.11.2}.  
	This completes the proof of the theorem.
\end{proof}

\begin{remark}\label{12.11.R1}
	Even if we obtain the solution $y \in L^2(0,T; D(\mathcal{A})) \cap H^1(0,T; H_0^1(\Omega; w))$ of equation \eqref{03.06.m} under the conditions $y^0 \in D(\mathcal{A}), y^1 \in H_0^1(\Omega; w)$, and $f = 0$, we still lack any information about $\frac{\partial y}{\partial \nu}$ near $0 \in \mathbb{R}^N$.
\end{remark}

\begin{lemma}\label{11.18.L1}
	The function
	\begin{equation*}
		y\in L^2(0,T; H_0^1(\Om;w))\cap H^1(0,T; L^2(\Om))\cap H^2(0,T; H^{-1}(\Om;w))
	\end{equation*}  
	is the weak solution of \eqref{03.06.m} with respect to $(y^0,y^1,f)$ if and only if 
	\begin{equation*}
		y\in L^2(0,T; H_0^1(\Om;w))\cap H^1(0,T; L^2(\Om))
	\end{equation*}
	satisfies
	\begin{equation}\label{11.18.3}
		\begin{split} 
			&\iint_Q y  \pt_{tt}\psi \df x\df t+\iint_Q w\nabla y\cdot \nabla \psi\df x\df t\\
			&=\iint_Q f\psi\df x\df t-\int_\Om y^0\pt_t\psi(0)\df x+\int_\Om y^1\psi(0)\df x
		\end{split}
	\end{equation}
	for all $\psi\in C^\iy(\ol Q)$ with $\supp\psi(t)\s \Om$ for all $t\in [0,T]$ and $\psi(T)=\pt_t\psi(T)=0$. 
\end{lemma}

\begin{proof}
	We prove the necessity. 
	
	Indeed, let $\psi\in C^\iy(\ol Q)$ with $\supp \psi(t)\s \Om$ for all $t\in [0,T]$ and $\psi(T)=\pt_t\psi(T)=0$. Since $y$ is a weak solution of \eqref{03.06.m} with respect to $(y^0, y^1, f)$, from Definition \ref{12.11.D1} (i),  we get
	\begin{equation}\label{03.04.2}
		y\in C([0,T]; L^2(\Om)), \mbox{ and } \pt_ty \in C([0,T]; H^{-1}(\Om;w)), 
	\end{equation}
	and 
	\begin{equation}\label{03.04.3}
		\lg \pt_{tt}y, \psi(t)\rg_{H^{-1}(\Om;w), H_0^1(\Om;w)}+(y, \psi(t))_{H_0^1(\Om;w)}=(f,\psi(t))_{L^2(\Om)}. 
	\end{equation}
	Integrating \eqref{03.04.3} on  $[0,T]$, we get
	\begin{equation*}
		\begin{split}
			\int_0^T\lg \pt_{tt}y, \psi(t)\rg_{H^{-1}(\Om;w), H_0^1(\Om;w)}\df t+\int_0^T(y,\psi(t))_{H_0^1(\Om;w)}\df t=\int_0^T(f,\psi(t))_{L^2(\Om)}\df t. 
		\end{split}
	\end{equation*}
	Integrating by parts with respect to $t\in [0,T]$, from \eqref{03.04.2}, we obtain 
	\begin{equation*}
		\begin{split}
			&\int_0^T(y,\pt_{tt}\psi)_{L^2(\Om)}\df t+\int_0^T(y,\psi)_{H_0^1(\Om;w)}\df t\\
			&=\int_0^T(f,\psi)_{L^2(\Om)}\df t-(y(0), \pt_t\psi(0))_{L^2(\Om)}+\lg \pt_t y(0), \psi(0)\rg_{H^{-1}(\Om;w), H_0^1(\Om;w)}\\
			&=\int_0^T(f,\psi)_{L^2(\Om)}\df t-(y^0, \pt_t\psi(0))_{L^2(\Om)}+(y^1,\psi(0))_{L^2(\Om)}. 
		\end{split}
	\end{equation*}
	This proves \eqref{11.18.3}. 
	
	Now, we prove the sufficiency.

		Indeed, it is clear that 
	\begin{equation*}
		\begin{split}
			\|\mcA y\|_{L^2(0,T; H^{-1}(\Om;w))}
			&=\sup_{ \|\psi\|_{L^2(0,T; H_0^1(\Om;w))}\leq 1}\int_0^T\lg \mcA y, \psi\rg_{H^{-1}(\Om;w), H_0^1(\Om;w)}\df t\\
			&=\sup_{  \|\psi\|_{L^2(0,T; H_0^1(\Om;w))}\leq 1}\iint_Q w\nabla y\cdot \nabla\psi\df x\df t\leq \|y\|_{L^2(0,T; H_0^1(\Om;w))}.
		\end{split}
	\end{equation*}
	i.e., $\mcA y\in L^2(0,T; H^{-1}(\Om;w))$. Now, taking $\xi\in C_0^\iy(Q)$, from  $f\in L^2(Q)\s L^2(0,T; H^{-1}(\Om;w))$ and  \eqref{11.18.3} we get 
	\begin{equation*}
		\begin{split}
			\iint_Q y\pt_{tt}\xi\df x\df t
			&=\iint_Q f\xi\df x\df t-\iint_Q w\nabla y\cdot \nabla \xi\df x\df t\\
			&=\int_0^T\lg f, \xi\rg_{H^{-1}(\Om;w),H_0^1(\Om;w)}\df x\df t-\int_0^T \lg \mcA y, \xi\rg_{H^{-1}(\Om;w), H_0^1(\Om;w)}\df t, 
		\end{split}
	\end{equation*}
	and hence 
	\begin{equation*}
		\pt_{tt}y=f-\mcA y
	\end{equation*}
	in the sense of distribution. Therefore, $\pt_{tt}y\in L^2(0,T; H^{-1}(\Om;w))$, i.e., 
	\begin{equation*}
		y\in H^2(0,T; H^{-1}(\Om;w)). 
	\end{equation*} 
	
	Fix $t\in (0,T)$ and $v\in C_0^\iy(\Om)$. Taking  $\de\in (0, \f{1}{4}\min\{t,T-t\})$. For arbitrary $\e\in (0,\f{1}{2}\de)$,  choose  $\zeta\in C_0^\iy(\R), 0\leq \zeta\leq 1$ such that $\zeta=1$ on $(t-\de,t+\de)$, and $\supp \zeta\s (t-\de-\e,t+\de+\e)$ and $\|\zeta-\chi_{(t-\de,t+\de)}\|_{L^2(\R)}<\e$, where $\chi_{(t-\de,t+\de)}$ is the characteristic function of the set $(t-\de,t+\de)$. Taking $\psi=\zeta v$, then $\psi\in C_0^\iy(Q)$, and from \eqref{11.18.3} we get (in distribution sense)
	\begin{equation*}
		\begin{split}
			\int_0^T \zeta \lg \pt_{tt}y, v\rg_{H^{-1}(\Om;w), H_0^1(\Om;w)}\df t+\iint_Q \zeta w\nabla y \cdot\nabla v\df x\df t=\iint_Q \zeta f v\df x\df t. 
		\end{split}
	\end{equation*}
	Hence, letting $\e\ra 0$, we get 
	\begin{equation*}
		\int_{t-\de}^{t+\de}\lg \pt_{tt}y, v\rg_{H^{-1}(\Om;w), H_0^1(\Om;w)}\df t+\iint_{\Om\ts (t-\de,t+\de)} w\nabla y \cdot  \nabla v\df x\df t=\iint_{\Om\ts (t-\de,t+\de)} f v\df x\df t.
	\end{equation*}
	Therefore, letting $\de\ra 0$, for a.e.\! $t\in (0,T)$ and each   $v\in C_0^\iy(\Om)$, we get 
	\begin{equation*}
		\lg \pt_{tt}y, v\rg_{H^{-1}(\Om;w), H_0^1(\Om;w)}+\int_\Om w\nabla y\cdot \nabla v\df x=\int_\Om fv\df x 
	\end{equation*} according to the Lebesgue Differentiation Theorem. This implies Definition \ref{12.11.D1} (i) since $C_0^\iy(\Om)$ is dense in $H_0^1(\Om;w)$. 
	
	Finally,  for each $v\in C_0^\iy(\Om)$ and for each $\zeta\in C^\iy(\R)$ such that $\zeta(T)=0, \pt_t\zeta(T)=0$, take $\psi(x,t)=\zeta(t)v(x)$, from \eqref{11.18.3} we have
	\begin{equation}\label{11.18.5}
		\begin{split} 
		&\int_0^T (y,v)_{L^2(\Om)} \pt_{tt}\zeta \df t+\int_0^T (y,v)_{H_0^1(\Om;w)} \zeta \df t \\
		&=\int_0^T (f,v)_{L^2(\Om)}\zeta \df t-(y^0,v)_{L^2(\Om)}\zeta(0)-(y^1,v)_{L^2(\Om)}\pt_t\zeta(0).
		\end{split} 
	\end{equation}
	Since $C_0^\infty(\Omega)$ is dense in $H_0^1(\Omega; w)$, identity \eqref{11.18.5} also holds for all $v \in H_0^1(\Omega; w)$. In particular, for all $n\in\N$, we obtain
	\begin{equation}\label{03.05.1}
		\begin{split}
			\int_0^T y_n(t)\pt_{tt}\zeta\df t+\la_n\int_0^Ty_n(t)\zeta \df t=\int_0^T f_n(t)\zeta\df t-y_n^0\zeta(0)-y_n^1\pt_t\zeta(0), 
		\end{split}
	\end{equation}
	where  $y_n(t)=(y(t),\Phi_n)_{L^2(\Om)}$, $f_n(t)=(f(t), \Phi_n)_{L^2(\Om)}, y_n^1=(y^1, \Phi_n)_{L^2(\Om)}$ and $y_n^0=(y^0, \Phi_n)_{L^2(\Om)}$. It is easily verified that $y_n(t)$ is a weak solution of the following one-dimensional system
	\begin{equation*} 
		\begin{cases}
			\pt_{tt}y_n(t)+\la_ny_n(t)=f_n(t), & t\in [0,T],\\
			y_n(0)=y_n^0, \pt_ty_n(0)=y_n^1, 
		\end{cases}
	\end{equation*}
	together with Steps 1--7 in the proof of Theorem \ref{12.11.T1}, we obtain that $y^m=\sum_{n=1}^m y_n(t)\Phi_n(x)$ converges weakly to the solution of \eqref{03.06.m} with respect to $(y^0,y^1,f)$ in $L^2(0,T; H_0^1(\Om;w))$ as $m\ra\iy$. Note that $y^m\ra \sum_{n=1}^\iy y_n\Phi_n=y$ strongly in $L^2(0,T; H_0^1(\Om;w))$ as $m\ra\iy$. Hence $y$ is the weak solution of \eqref{03.06.m} with respect to $(y^0,y^1,f)$. Moreover, $y(0)=y^0$ and $\pt_ty(0)=y^1$, i.e., Definition \ref{12.11.D1} (ii) holds. This completes the proof of the lemma. 
\end{proof}

\section{Shape-Design Approximation}\label{S3}

In this section, we introduce the shape-design approximation for \eqref{03.06.m}. The idea is to replace the original domain, whose boundary contains the degenerate point, by a family of regularized domains obtained by removing a small neighborhood of that point. This produces a sequence of uniformly hyperbolic problems on smooth domains, which can then be compared with the original degenerate equation. The goal of the section is to prove that the regularized solutions converge to the degenerate solution in the natural energy topology and, away from the degenerate point, at the level of boundary normal derivatives.
Let $\al\in (0,1), \e\in (0,\f{1}{8}R_0)$ and $\Om_\e$ be defined in Assumption \ref{Assumption (H)}. 

Now, we consider the following equation
\begin{equation}\label{12.12.15}
	\begin{cases}
		\pt_{tt}y_\e-\Div(|x|^\al\nabla y_\e)=f_\e, &\mbox{in } Q_\e,\\
		y_\e=0, &\mbox{on }\pt Q_\e,\\
		y_\e(0)=y_\e^0, \pt_ty_\e(0)=y_\e^1, &\mbox{in }\Om_\e,
	\end{cases}
\end{equation}
where $Q_\e=\Om_\e\ts (0,T)$, and $y_\e^0\in H_0^1(\Om_\e;w_\e)$ (see \eqref{12.14.2} below), $y_\e^1\in L^2(\Om_\e)$ and $f_\e\in L^2(Q_\e)$. Here and in what follows, we denote 
\begin{equation}\label{12.14.3}
	w_\e=w|_{\Om_\e}, \quad \mcA_\e u=\Div(w_\e\nabla u) \mbox{ is defined on  } \Om_\e. 
\end{equation}
It is clear that \eqref{12.12.15} is a uniformly hyperbolic equation. 

Define
\begin{equation*}
	H^1(\Om_\e; w_\e)=\left\{u\in L^2(\Om_\e)\colon \int_{\Om_\e} w_\e\nabla u\cdot\nabla u\df x<+\iy\right\}. 
\end{equation*}
Its inner product and norm are defined by 
\begin{equation*}
	(u,v)_{H^1(\Om_\e;w_\e)}=\int_{\Om_\e} w_\e\nabla u\cdot\nabla v\df x,\quad \|u\|_{H^1(\Om;w_\e)}=(u,u)_{H^1(\Om_\e; w_\e)}^\f{1}{2}. 
\end{equation*}
Set 
\begin{equation}\label{12.14.2}
	H_0^1(\Om_\e; w_\e)=\mbox{the closure of $C_0^\iy(\Om_\e)$ in } H^1(\Om_\e;w_\e). 
\end{equation}

Define 
\begin{equation*}
	H^2(\Om_\e; w_\e)=\left\{u\in H^1(\Om_\e;w_\e)\colon \mcA_\e u\in L^2(\Om_\e)\right\}.
\end{equation*}
Its inner product and norm are defined by 
\begin{equation*}
	\begin{split}
		(u,v)_{H^2(\Om_\e; w_\e)}=(u,v)_{H^1(\Om_\e;w_\e)}+(\mcA_\e u, \mcA_\e v)_{L^2(\Om_\e)},\quad \|u\|_{H^2(\Om_\e;w_\e)}=(u,u)_{H^2(\Om_\e; w_\e)}^\f{1}{2}.
	\end{split}
\end{equation*}
Set 
\begin{equation*}
	D(\mcA_\e)=H^2(\Om_\e; w_\e)\cap H_0^1(\Om_\e;w_\e). 
\end{equation*}

It is clear that 
\begin{equation*}
	H_0^1(\Om_\e;w_\e)=H_0^1(\Om_\e),\ H^1(\Om_\e;w_\e)=H^1(\Om_\e), \mbox{ and } H^2(\Om_\e;w_\e)=H^2(\Om_\e)
\end{equation*}
for all $\e\in (0,\f{1}{8}R_0)$. Hence $H_0^1(\Om_\e;w_\e)$, $H^1(\Om_\e)$, and $H^2(\Om_\e)$ are Hilbert spaces.

\begin{lemma}\label{12.12.L1}
	Let $N \geq 2$ and $\alpha \in (0, 1)$. Then, for all $u \in H_0^1(\Omega_\e; w_\e)$, the following inequality holds:
	\begin{equation*}
		(N - 2 + \alpha) \left\||x|^{\frac{\alpha}{2} - 1} u\right\|_{L^2(\Omega_\e)} \leq 2 \|\nabla u\|_{L^2(\Omega_\e; w_\e)}.
	\end{equation*}
	Furthermore, if $u \in H_0^1(\Omega; w_\e)$, it follows that $u \in L^2(\Omega_\e)$. 
\end{lemma}

\begin{proof}
	The proof is the same as that of Lemma \ref{08.15.L1}. 
\end{proof}

\begin{remark}\label{12.12.R1}
	From Lemma \ref{12.12.L1} and Assumption \ref{Assumption (H)},  we have
	\begin{equation}\label{12.12.16} 
		 \int_{\Om_\e} u^2\df x \leq \f{4M^{2-\al}}{(N-2+\al)^2}\int_{\Om_\e} w_\e\nabla u\cdot \nabla u\df x,
	\end{equation}
	for all $u\in H_0^1(\Om_\e;w_\e)$. This is the Poincar\'{e} inequality.
	In particular, the norm
	\begin{equation}\label{12.12.17}
		\|u\|_{H_0^1(\Omega_\e; w_\e)} = \left(\int_{\Omega_\e} (\nabla u \cdot \nabla u) w_\e \, \mathrm{d}x\right)^{\frac{1}{2}}
	\end{equation}
	is an equivalent norm in $H_0^1(\Omega_\e; w_\e)$. Hereafter, we use \eqref{08.19.2} to define the norm of $H_0^1(\Omega_\e; w_\e)$.
\end{remark}

\begin{lemma}\label{03.05.L2}
	The embedding $H_0^1(\Omega_\e; w_\e) \hookrightarrow L^2(\Omega_\e)$ is compact.
\end{lemma}

\begin{proof}
	This follows from the classical Sobolev compact embedding, since $H_0^1(\Om_\e;w_\e)=H_0^1(\Om_\e)$. 
\end{proof}

\begin{notation}\label{03.05.N1}
	From Remark \ref{12.12.R1} and Lemma \ref{03.05.L2} we obtain that the degenerate partial differential operator $\mcA_\e$ has discrete point spectrum 
	\begin{equation}\label{03.05.2}
		0<\la_1^\e< \la_2^\e\leq \la_3^\e\leq  \cdots \ra +\iy, 
	\end{equation}
	i.e., $\la_n^\e\ (n\in\N)$ satisfies the following equation
	\begin{equation}\label{03.05.3}
		\begin{cases}
			\mcA_\e \Phi_n^\e=\la_n ^\e\Phi_n^\e, & \mbox{in }\Om_\e, \\
			\Phi_n^\e=0, &\mbox{on }\pt\Om_\e.
		\end{cases}
	\end{equation}
	Moreover, from \eqref{08.16.8},  we get
	\begin{equation}\label{03.05.4}
		\la_1^\e=\inf_{0\neq u\in H_0^1(\Om_\e;w_\e)}\f{\int_{\Om_\e} w_\e\nabla u\cdot \nabla u\df x}{\int_{\Om_\e} u^2\df x}\geq \f{(N-2+\al)^2}{4M^{2-\al}}.
	\end{equation}
	We denote $\Phi_n^\e(x)$ the $n$th eigenfunction of $\mcA_\e$ with respect to the eigenvalue $\la_n^\e\ (n\in\N)$, and $\{\Phi_n^\e\}_{n\in\N}$ is the orthonormal basis of $L^2(\Om_\e)$, moreover, $\{\Phi_n^\e\}_{n\in\N}$ is an orthogonal subset of $H_0^1(\Om_\e;w_\e)$. (See \cite[Theorem 7 (pp. 728) in Appendix D]{Evans}, or the proof of Lemma \ref{06.28.L1} in the following.) 
\end{notation}

\begin{lemma}\label{03.05.L1}
	Let $u=\sum_{i=1}^\iy u_i \Phi_i^\e\in H_0^1(\Om_\e;w_\e)$ with $u_i=(u,\Phi_i^\e)_{L^2(\Om_\e)}$ for all $i\in\N$. We have $\nabla u=\sum_{i=1}^\iy u_i\nabla \Phi_i^\e$ and $\|u\|_{H_0^1(\Om_\e;w_\e)}=(\sum_{i=1}^\iy u_i^2\la_i^\e)^\f{1}{2}$, and
	\begin{equation*}
		u\in H^2(\Om_\e;w_\e)\Lra \sum_{i=1}^\iy u_i^2(\la_i^\e)^2<\iy,
	\end{equation*}
	and
	\begin{equation*}
		\mcA_\e  u=\sum_{i=1}^\iy u_i\la_i^\e\Phi_i^\e \mbox{ and } \|\mcA_\e u\|_{L^2(\Om_\e)}=\left(\sum_{i=1}^\iy u_i^2(\la_i^\e)^2\right)^\f{1}{2}.
	\end{equation*}
\end{lemma}

\begin{definition}\label{12.12.D1}
	We say that 
	\begin{equation*}
		y_\e\in L^2(0,T; H_0^1(\Om_\e;w_\e))\cap H^1(0,T; L^2(\Om_\e))\cap H^2(0,T; H^{-1}(\Om_\e;w_\e))
	\end{equation*}
	is a weak solution of \eqref{12.12.15} with respect to $(y_\e^0,y_\e^1,f_\e)$ if
	
	(i) for every $v\in H_0^1(\Om_\e;w_\e)$ and a.e.~$t\in (0,T)$ we have 
	\begin{equation*}
		\lg\pt_{tt}y_\e, v\rg_{H^{-1}(\Om_\e;w_\e),H_0^1(\Om_\e;w_\e)}+\int_{\Om_\e} w_\e\nabla y_\e\cdot \nabla v\df x=\int_{\Om_\e} f_\e v\df x, 
	\end{equation*}
	
	(ii) $y_\e(0)=y_\e^0$ and $\pt_ty_\e(0)=y_\e^1$. 
\end{definition}

Since \eqref{12.12.15} is a uniformly hyperbolic equation, it admits a unique weak solution. However, the a priori estimates depend on the coefficients involving $w_\e$ and $\Om_\e$, and thus on the parameter $\e \in (0, \tfrac{1}{8}R_0)$. In order to obtain estimates that are uniform in $\e$, we reconstruct the argument in a way similar to Theorem \ref{12.11.T1}, which yields the following theorem.

\begin{theorem}\label{12.12.T1}
	There exists a unique weak solution $y_\e$ to the equation \eqref{12.12.15}, moreover, we have 
	\begin{equation}\label{12.12.18}
		\begin{split}
			&\esssup_{t\in (0,T)}\left(\|y_\e(t)\|_{H_0^1(\Om_\e;w_\e)}+\|\pt_ty_\e(t)\|_{L^2(\Om_\e)}\right)+\|\pt_{tt}y_\e\|_{L^2(0,T; H^{-1}(\Om_\e;w_\e))}\\
			&\hspace{4.5mm}+\left\|\f{\pt y_\e}{\pt \nu}\right\|_{L^2(0,T; L^2(\pt\Om_\e-B(0,R_0)))}\\
			&\leq C\left(\|y_\e^0\|_{H_0^1(\Om_\e;w_\e)}+\|y_\e^1\|_{L^2(\Om_\e)}+\|f_\e\|_{L^2(Q_\e)}\right), 
		\end{split}
	\end{equation}
	where the constant $C>0$ depends only on $\al, R_0, T$ and $\Om-B(0,\f{1}{2}R_0)$.

	Furthermore, if in addition $y_\e^0\in D(\mcA_\e)$ and $y_\e^1\in H_0^1(\Om_\e;w_\e)$ and $f_\e\in H^1(0,T; L^2(\Om_\e))$, then 
	\begin{equation*}
		y_\e\in L^2(0,T; D(\mcA_\e))\cap H^1(0,T; H_0^1(\Om_\e;w_\e))\cap H^2(0,T; L^2(\Om_\e)), 
	\end{equation*}
	moreover, we have 
	\begin{equation}\label{12.12.19}
		\begin{split}
			&\esssup_{t\in (0,T)}\left(\|y_\e(t)\|_{H^2(\Om-B(0,R_0))}+\|y_\e(t)\|_{D(\mcA_\e)}+\|\pt_ty_\e(t)\|_{H_0^1(\Om_\e;w_\e)}+\|\pt_{tt}y_\e(t)\|_{L^2(\Om_\e)}\right)\\
			&\hspace{4.5mm}+\left\|\f{\pt y_\e}{\pt \nu}\right\|_{H^1(0,T; L^2(\pt\Om_\e-B(0,R_0)))}\\
			&\leq C\left(\|y_\e^0\|_{D(\mcA_\e)}+\|y_\e^1\|_{H_0^1(\Om_\e;w_\e)}+\|f\|_{H^1(0,T; L^2(\Om_\e))}\right). 
		\end{split}
	\end{equation}
	Here the constants $C>0$ depend only on $\al, R_0, T$ and $\Om-B(0,\f{1}{2}R_0)$. 
\end{theorem}

\begin{proof}
	The proof is similar to that of Theorem \ref{12.11.T1}.
	
	Note that the a priori estimates \eqref{12.12.18} and \eqref{12.12.19} are analogous to the estimates \eqref{12.11.1} and \eqref{12.11.2}, respectively. More precisely, the estimates \eqref{12.11.1} and \eqref{12.11.2} are obtained through \eqref{12.12.2}, \eqref{12.12.3}, \eqref{12.12.12}, \eqref{09.07.16}, and \eqref{12.12.14}. In deriving these estimates, we only use Gronwall's inequality, Remark \ref{08.16.R1}, Lemma \ref{06.28.L1}, Lemma \ref{11.30.L2}, \cite[Theorem 2 (p.~302) in Chapter 5.9.2]{Evans}, and \cite[Theorem 1 (p.~327) in Chapter 6.3.1 and Theorem 4 (p.~334) in Chapter 6.3.2]{Evans}.
	
	By replacing Remark \ref{08.16.R1} with Remark \ref{12.12.R1}, Lemma \ref{08.15.L4} with Lemma \ref{03.05.L2}, and Lemma \ref{06.28.L1} with Lemma \ref{03.05.L1}, while noting that the remaining arguments are independent of $\e \in (0,\frac{1}{8}R_0)$, we obtain the estimates \eqref{12.12.18} and \eqref{12.12.19}.
	
	This completes the proof of the theorem.
\end{proof}

\begin{lemma}\label{12.13.L1}
	The function
	\begin{equation*}
		y_\e\in L^2(0,T; H_0^1(\Om_\e;w_\e))\cap H^1(0,T; L^2(\Om_\e))\cap H^2(0,T; H^{-1}(\Om_\e; w_\e))
	\end{equation*}  
	is the weak solution of \eqref{12.12.15} with respect to $(y_\e^0,y_\e^1, f_\e)$ if and only if 
	\begin{equation*}
		y_\e\in L^2(0,T; H_0^1(\Om_\e;w_\e))\cap H^1(0,T; L^2(\Om_\e))
	\end{equation*}
	satisfies
	\begin{equation}\label{11.21.1}
		\begin{split} 
			&\iint_{Q_\e} y_\e(t)  \pt_{tt}\psi(t) \df x\df t+\iint_{Q_\e} w_\e\nabla y_\e\cdot \nabla \psi\df x\df t\\
			&=\iint_{Q_\e} f_\e\psi\df x\df t+\int_{\Om_\e} y_\e^1\psi(0)\df x-\int_{\Om_\e} y_\e^0\pt_t\psi(0)\df x
		\end{split}
	\end{equation}
	for all $\psi\in C^\iy(\ol Q_\e)$ with $\supp\psi(t)\s \Om_\e$ for all $t\in [0,T]$ and $\psi(T)=\pt_t\psi(T)=0$. 
\end{lemma}

\begin{proof}
	The proof is the same as that of Lemma \ref{11.18.L1}. 
\end{proof}

We are now in a position to employ the shape design method for equation \eqref{03.06.m}, based on the regularized problem \eqref{12.12.15} with $\e\in (0,\frac{1}{8}R_0)$. 

\begin{notation} \label{03.05.N2}
Let $\e\in (0,\f{1}{8}R_0)$, and $y_\e$ is the weak solution of \eqref{12.12.15} with respect to $(y_\e^0, y_\e^1, f_\e)$. We extend $y_\e$ to 
\begin{equation}\label{12.12.21}
	E y_\e=
	\begin{cases}
		y_\e, &\mbox{in } Q_\e, \\
		0, &\mbox{on } Q-Q_\e. 
	\end{cases}
\end{equation}
Then we have
\begin{equation}\label{12.13.2}
	E\pt_ty_\e=\pt_t Ey_\e,
\end{equation}
indeed, for each $\zeta\in C_0^\iy(0,T)$ and $v\in C_0^\iy(\Om)$, we have
\begin{equation*}
	\begin{split}
		\iint_{Q}\left[\pt_tEy_\e\right]\zeta v\df x\df t
		&=-\iint_{Q}(Ey_\e)(\pt_t\zeta)v\df x\df t=-\iint_{Q_\e} y_\e (\pt_t\zeta)v\df x\df t\\
		&=\iint_{Q_\e} (\pt_ty_\e)\zeta v\df x\df t=\iint_{Q}[E(\pt_ty_\e)]\zeta v\df x\df t,  
	\end{split}
\end{equation*}
this implies that $\pt_tEy_\e=E\pt_ty_\e$ in the sense of distribution, moreover, we have $\pt_tEy_\e=E\pt_ty_\e$. 
And from 
\begin{equation*}
	\nabla Ey_\e=
	\begin{cases}
		\nabla y_\e, &\mbox{on }Q_\e,\\
		0, &\mbox{on }Q-Q_\e
	\end{cases}
\end{equation*}
and 
\begin{equation*}
	\iint_{Q} w\nabla Ey_\e\cdot \nabla Ey_\e\df x\df t+\iint_Q (Ey_\e)^2\df x\df t=\iint_{Q_\e} w_\e\nabla y_\e\cdot \nabla y_\e\df x\df t+\iint_{Q_\e} (y_\e)^2\df x\df t
\end{equation*}
we get 
\begin{equation}\label{12.13.1}
	Ey_\e\in L^2(0,T; H_0^1(\Om;w))\cap H^1(0,T; L^2(\Om)).
\end{equation}
Moreover, we have
\begin{equation*}
	\|Ey_\e\|_{L^2(0,T; H_0^1(\Om;w))}=\|y_\e\|_{L^2(0,T; H_0^1(\Om_\e;w_\e))}. 
\end{equation*}
\end{notation}

Theorem \ref{12.12.T2}  is one of the main results of this paper. It asserts that the solution of equation \eqref{03.06.m} can be approximated, via the shape design method, by the solutions of equation \eqref{12.12.15} for $\e \in (0,\frac{1}{8}R_0)$. Now, we prove this theorem. 

We denote  $\dist(A, B)=\inf_{x\in A, z\in B}|x-z|$ for any sets $A, B\s\R^N$.

\begin{proof}[Proof of Theorem \ref{12.12.T2}]
	From Assumption \ref{Assumption (H)} (i.e., $\pt\Om_\e-B(0,R_0)=\pt\Om-B(0,R_0)$ for all $\e\in (0,\f{1}{8}R_0)$), it suffices to consider $\e\in (0,\min\{\f{1}{8}R_0, \dist(\supp y^0, 0)\})$. Then, from \eqref{12.12.18} and Notation \ref{03.05.N2}, we get 
	\begin{equation*}
		\begin{split}
			&\|Ey_\e\|_{L^2(0,T; H_0^1(\Om_\e;w_\e))}+\|\pt_tEy_\e\|_{L^2(Q_\e)}+\left\|\f{\pt Ey_\e}{\pt\nu}\right\|_{L^2(0,T; L^2(\pt\Om-B(0,R_0)))}\\
			&\leq C\left(\|y^0\|_{H_0^1(\Om;w)}+\|y^1\|_{L^2(\Om)}+\|f\|_{L^2(Q)}\right),  
		\end{split}
	\end{equation*}
	where the positive constant $C$ depends only on $\al, R_0, T$, and $\Om-B(0,\f{1}{2}R_0)$. Therefore, there exist $\vp\in L^2(0,T; H_0^1(\Om;w))\cap H^1(0,T; L^2(\Om))$ and $z\in L^2(0,T; L^2(\pt\Om-B(0,R_0)))$ such that 
	\begin{equation}\label{12.13.4}
		\begin{split}
			Ey_\e
			&\ra \vp\ \!\quad  \mbox{ weakly in } L^2(0,T; H_0^1(\Om;w)), \\
			\pt_t Ey_\e 
			&\ra \pt_t\vp\   \mbox{ weakly in } L^2(Q),\\
			\f{\pt Ey_\e}{\pt \nu}
			&\ra z\ \quad \mbox{ weakly in } L^2(0,T; L^2(\pt\Om-B(0,R_0))).
		\end{split}
	\end{equation}
	
	Now, we prove $\vp=y$ and $z=\f{\pt Ey_\e}{\pt \nu}$. 
	
	Let $\psi\in C^\iy(\ol Q)$ with $\supp \psi(t)\s \Om$ for all $t\in [0,T]$ and $\psi(T)=\pt_t\psi(T)=0$. Then,  for all $\e\in (0, \f{1}{2}\min_{t\in [0,T]} \dist(\supp\psi(t), \pt\Om))$, we have $\supp \psi(t)\s\Om_\e$ for all $t\in [0,T]$. Hence,  from Lemma \ref{12.13.L1},  we get
	\begin{equation*}
		\begin{split}
			&\iint_{Q_\e} y_\e(t)  \pt_{tt}\psi(t) \df x\df t+\iint_{Q_\e} w_\e\nabla y_\e\cdot \nabla \psi\df x\df t\\
			&=\iint_{Q_\e} f_\e\psi\df x\df t+\int_\Om y_\e^1\psi(0)\df x-\int_{\Om_\e} y_\e^0\pt_t\psi(0)\df x. 
		\end{split}
	\end{equation*}
	Extending $y_\e$ to $Ey_\e$, note that $\supp \psi(t)\s \Om_\e$ for all $t\in [0,T]$, we obtain
	\begin{equation*}
		\begin{split}
			&\iint_{Q} Ey_\e(t)  \pt_{tt}\psi(t) \df x\df t+\iint_{Q} w\nabla Ey_\e\cdot \nabla \psi\df x\df t\\
			&=\iint_{Q} f\psi\df x\df t+\int_\Om y^1\psi(0)\df x-\int_{\Om} y^0\pt_t\psi(0)\df x.
		\end{split}
	\end{equation*}
	From \eqref{12.13.4} we obtain 
	\begin{equation*}
		\begin{split}
			&\iint_{Q} \vp(t)  \pt_{tt}\psi(t) \df x\df t+\iint_{Q} w\nabla \vp\cdot \nabla \psi\df x\df t=\iint_{Q} f\psi\df x\df t+\int_\Om y^1\psi(0)\df x-\int_{\Om} y^0\pt_t\psi(0)\df x.
		\end{split}
	\end{equation*}
	This implies that $\vp$ is a solution of \eqref{03.06.m} by Lemma \ref{11.18.L1}. Hence $\vp=y$ by Theorem \ref{12.11.T1} (uniqueness). 
	
	Next, let $\xi\in C_0^\iy([\pt\Om-B(0,R_0)]\ts (0,T))$. Extend it to a function $\wt\xi\in C^\iy(\ol Q)$ such that $\wt \xi=\xi$ on $[\pt\Om-B(0,R_0)]\ts [0,T]$, $\wt \xi=0$ on $[\pt\Om-\supp\xi]\ts [0,T]$, $\wt\xi=0$ in $[\Om\cap B(0,\f{1}{2}R_0)]\ts [0,T]$, and $\wt \xi=\pt_t\wt \xi=0$ at $t=0,T$. Multiplying \eqref{12.12.15} by $\wt\xi$, and using that \eqref{12.12.15} is uniformly hyperbolic, integration by parts gives
	\begin{equation*}
		\begin{split}
			&-\iint_{\pt Q_\e} \xi w_\e \nabla y_\e\cdot \nu\df S\df t-\iint_{Q_\e} (\pt_ty_\e)(\pt_t\wt \xi)\df x\df t+\iint_{Q_\e} w_\e\nabla y_\e\cdot \nabla\wt \xi \df x\df t=\iint_{Q_\e} f_\e \wt \xi\df x\df t
		\end{split}
	\end{equation*}
	when $\e\in (0,\f{1}{2}R_0)$. 
	Extending $y_\e$ to $Ey_\e$, from \eqref{12.13.4},  letting $\e\ra 0$, we get
	\begin{equation}\label{12.13.5}
		-\iint_{\pt Q} \xi w z\df S\df t-\iint_Q (\pt_ty)(\pt_t\wt \xi)\df x\df t+\iint_Qw\nabla y\cdot \nabla\wt \xi\df x\df t=\iint_Q f\wt \xi\df x\df t. 
	\end{equation}
	Note that the first term in \eqref{12.13.5} is indeed $\iint_{\pt Q}\xi wz\df S\df t=\iint_{[\pt\Om-B(0,\f{1}{2}R_0)]\ts (0,T)} \xi wz\df S\df t$. 
	Multiplying $\xi$ on the both sides of \eqref{03.06.m}, note that $\wt\xi =0$ on $[\Om\cap B(0,\f{1}{2}R_0)]\ts [0,T]$ and $y(t)\in H^2(\Om-B(0,R_0))$ for a.e. $t\in (0,T)$, integrating by parts, from Theorem \ref{12.11.T1} \eqref{12.11.2}, we have
	\begin{equation*}
		\begin{split}
			-\iint_{\pt Q} \xi w \f{\pt y}{\pt\nu}\df S\df t-\iint_Q (\pt_ty)(\pt_t\wt \xi)\df x\df t+\iint_Qw\nabla y\cdot \nabla\wt \xi\df x\df t=\iint_Q f\wt \xi\df x\df t. 
		\end{split}
	\end{equation*}
	Combining this with \eqref{12.13.5}, we obtain 
	\begin{equation*}
		z=\f{\pt y}{\pt \nu}. 
	\end{equation*}
	
	Finally, let $y^0,y^1\in C_0^\iy(\Om)$ and $f=0$. Then it is clear that $y^0\in D(\mcA_\e)$ and $y^1\in H_0^1(\Om_\e;w_\e)$ for all $\e\in \left(0,\min\left\{\f{1}{8}R_0, \dist(\supp y^0,0), \dist(\supp y^1,0)\right\}\right)$.
	On the one hand, from the fifth term in \eqref{12.12.19} we know that $\{\pt_t\f{\pt y_\e}{\pt \nu}=\f{\pt (\pt_ty_\e)}{\pt \nu}\}$ is a bounded set in $L^2(0,T; L^2(\pt\Om-B(0,R_0)))$. On the other hand, from the first term in \eqref{12.12.19}, the classical Sobolev trace theorem, and $\Om_\e-B(0,R_0)=\Om-B(0,R_0)$ for all $\e\in(0,\f{1}{8}R_0)$, we get 
	\begin{equation*}
		\begin{split}
			\left\|\f{\pt y_\e}{\pt\nu}\right\|_{L^2(0,T; H^\f{1}{2}(\pt\Om-B(0,R_0)))}\leq C\|y_\e\|_{L^2(0,T; H^2(\Om-B(0,R_0)))}\leq C\left(\|y^0\|_{D(\mcA)}+\|y^1\|_{H_0^1(\Om;w)}\right), 
		\end{split}
	\end{equation*}
	where the positive constant $C$ depends only on $\al, R_0, T$, and $\Om-B(0,\f{1}{2}R_0)$. 
	Combining this with the embeddings $H^\f{1}{2}(\pt\Om-B(0,R_0))\hra L^2(\pt\Om-B(0,R_0))$ and 
	\begin{equation*}
		\begin{split} 
		&\left\{\f{\pt y_\e}{\pt \nu}\in L^2(0,T; H^\f{1}{2}(\pt\Om-B(0,R_0)))\colon \pt_t\f{\pt y_\e}{\pt\nu}=\f{\pt (\pt_ty_\e)}{\pt\nu}\in L^2(0,T; L^2(\pt\Om-B(0,R_0)))\right\}\\
		&\hra L^2(0,T; L^2(\pt\Om-B(0,R_0)))
		\end{split} 
	\end{equation*}
	are compact, we obtain \eqref{12.13.3} along a subsequence. This completes the proof of this theorem.
\end{proof}

\section{Observability via Shape Design}\label{S4}

In this section, we establish observability for the degenerate equation \eqref{03.06.m} by combining the shape-design approximation with a multiplier argument on the regularized domains. We first work on each approximate problem posed on $\Omega_\e$, where the equation is uniformly hyperbolic and the multiplier computation is classical. We then pass to the limit $\e\to 0$ by using the convergence results obtained in Section \ref{S3}.

Let $\e\in (0,\f{1}{8}R_0)$. In this section we use standard notation from differential geometry. In particular, we continue to write $\pxi$ for $\f{\pt}{\pt x_i}$ and use $\de_{ij}, \de^{ij}$, or $\de_i^j$ for the Kronecker symbols.

\begin{notation}\label{03.05.N3}
	Let $\R^n$ have the usual topology and $x=(x^1,\cdots, x^N)$ be the natural coordinate system. Denote  the usual product and norm on $\R^N$ 
	\begin{equation*}
		X\cdot Y=\lg X, Y\rg_0=\sum_{i=1}^N X^iY^i, \quad |X|_0=(X\cdot X)^\f{1}{2}, 
	\end{equation*}
	where $X=\sum_{i=1}^NX^i\pxi, Y=\sum_{i=1}^N Y^i\pxi\in\R_x^N$ and $\R_x^N$ is the tangent space at the point $x\in\R^N$.  
	Let $X=\sum_{i=1}^N X^i\pxi$ be a vector field on $\R^N$, denote the divergence of $X$ in the Euclidean metric by 
	\begin{equation*}
		\Div_0(X)=\sum_{i=1}^N\f{\pt X^i}{\pt x_i}. 
	\end{equation*}
	For each $f\in C^1(\ol\Om)$, denote the gradient of $f$ by 
	\begin{equation}\label{11.26.2}
		\nabla f=\nabla_0 f=\sum_{i=1}^N \f{\pt f}{\pt x^i} \pxi. 
	\end{equation}
\end{notation}

\begin{notation}\label{03.05.N4}
	Set 
	\begin{equation}\label{11.23.1}
		g_\e(x)=(g_{ij}^\e)_{i,j=1,\cdots, N}=w_\e^{-1}I_N\mbox{ for all } x\in \Om_\e,    
	\end{equation}
	i.e., $g_{ij}^\e=w_\e^{-1}\de_{ij}=[|x|^{-\al}]|_{\Om_\e}\de_{ij}$ for all $i,j=1,\cdots, N$, 
	where  $\de_{ij}$ is the Kronecker symbol, and $I_N$ is the identity matrix. Then 
	\begin{equation}\label{11.23.2}
		g_\e^{-1}=(g_\e^{ij})_{i,j=1,\cdots, N}=w_\e I_N=A_\e I_N, 
	\end{equation}
	i.e., $g_\e^{ij}=w_\e \de^{ij}$ for all $i,j=1,\cdots, N$.  
	
	For each $x\in\Om_\e$, define the inner product and norm over the tangent space $(\Om_\e)_x=\R^N$ by 
	\begin{equation}\label{11.25.7}
		g_\e(X,Y)=\lg X, Y\rg_{g_\e}=\sum_{i,j=1}^N g_{ij}^\e X^iY^j,\quad  |X|_{g_\e}=\lg X, X\rg_{g_\e}^\f{1}{2}=\left(\sum_{i=1}^N w_\e^{-1}(X^i)^2\right)^\f{1}{2}, 
	\end{equation} 
	where $X=\sum_{i=1}^NX^i\pxi, Y=\sum_{i=1}^NY^i\pxi\in\Om_x$. Moreover, 
	\begin{equation}\label{11.24.5}
		\llg \pxi, \pxj\rrg_{g_\e}=w_\e^{-1}\de_{ij} \mbox{ for all } i,j=1,\cdots, N,
	\end{equation} 
	and  
	\begin{equation}\label{11.28.1}
		\lg X, A_\e(x)Y\rg_{g_\e}=X\cdot Y \mbox{ for all } x\in\R^N. 
	\end{equation}
	
	We directly verify that  $(\Om_\e, g_\e)$ is a $C^2$ Riemannian manifold. Its Christoffel symbol $\Ga_{ij}^k(\e)=\f{1}{2}\sum_{l=1}^N g_\e^{kl}(\f{\pt g_{li}^\e}{\pt x^j}+\f{\pt g_{lj}^\e}{\pt x^i}-\f{\pt g_{ij}^\e}{\pt x^l})\ (i,j,k=1,\cdots, N)$ is 
	\begin{equation}\label{11.24.9}
		\Ga_{ij}^k(\e)=-\f{\al}{2}|x|^{-2}\left(x^j\de_{ik}+x^i\de_{jk}-x^k\de_{ij}\right)
	\end{equation}
	by \eqref{12.14.3} and \eqref{11.23.2} and Assumption \ref{Assumption (H)}. 
	
	The gradient of a function $u$ on  $(\Om_\e,g_\e)$ is given by 
	\begin{equation}\label{11.24.3}
		\nabla_{g_\e} u=\sum_{i=1}^N g_\e^{ij}\f{\pt u}{\pt x^i}\f{\pt}{\pt x^j}=\sum_{i=1}^N w_\e \f{\pt u}{\pt x^i}\pxi, 
	\end{equation}
	then from \eqref{11.24.5} we have 
	\begin{equation}\label{11.24.4}
		\lg \nabla_{g_\e}u, \nabla_{g_\e}v\rg_{g_\e}=\sum_{i,j=1}^N w_\e^2\f{\pt u}{\pt x^i}\f{\pt v}{\pt x^j}\llg \pxi,\pxj\rrg_{g_\e}=\sum_{i=1}^N w_\e\f{\pt u}{\pt x^i}\f{\pt v}{\pt x^i}=(\nabla_0 u\cdot \nabla_0v)w_\e. 
	\end{equation}
	
	The divergence of a vector field $X=\sum_{i=1}^N X^i\pxi$ on $(\Om_\e,g_\e)$ is given by 
	\begin{equation*}
		\begin{split}
			\Div_{g_\e} X=\sum_{i=1}^N\f{\pt X^i}{\pt x^i}+\sum_{i,k=1}^N X^i\Ga_{ik}^k(\e)=\sum_{i=1}^N \f{\pt X^i}{\pt x^i}-\f{\al N}{2} |x|^{-2}\sum_{i=1}^NX^ix^i.
		\end{split}
	\end{equation*}
	Hence 
	\begin{equation}\label{11.24.1}
		\mcA u=-\sum_{i,j=1}^N \pxi\left(w_\e(x)\f{\pt u}{\pt x_j}\right)=-\Div_0(\nabla_{g_\e}u).
	\end{equation} 
	
	Denote the Levi-Civita connection in metric $g_\e$ by $D$, i.e., 
	\begin{equation}\label{11.24.8}
		D_{\pxi}\pxj=\sum_{k=1}^N\Ga_{ij}^k(\e) \pxk \mbox{ for all } i,j, k=1,\cdots, N. 
	\end{equation}
	In other words,  for each vector field $H$ on Riemannian manifold $(\Om_\e,g_\e)$, we have 
	\begin{equation}\label{11.24.7}
		DH(X,Y)=\lg D_X H, Y\rg_{g_\e} \mbox{ for all } X,Y\in(\Om_\e)_x. 
	\end{equation}
\end{notation}

\begin{remark}\label{11.24.R1}
	There are two metrics on $\Om_\e\subset\R^N$: the standard Euclidean metric and the Riemannian metric $g_\e$ for fixed $0<\e<\f{1}{8}R_0$. In what follows, we write $\nabla f=\nabla_0 f$ for the Euclidean gradient, while $\nabla_{g_\e}$ denotes the gradient on the Riemannian manifold $(\Omega_\e,g_\e)$.
\end{remark}

\subsection{Observability for the approximate equations}

In this subsection, we prove an observability estimate for the approximate equation \eqref{12.12.15}. Since the approximate problems are uniformly hyperbolic for all $\e\in (0,\f{1}{8}R_0)$, all integrations by parts are justified on $\Omega_\e$. This is precisely the technical advantage of the shape-design procedure.

\begin{lemma}\label{11.24.L1}
	Let $H$ be a $C^1$ vector field on the Riemannian  manifold $(\Om_\e,g_\e)$, let $f\in C^2(\Om_\e)$. Then 
	\begin{equation*}
		\begin{split}
			&\lg \nabla_{g_\e} f, \nabla_{g_\e} (H(f))\rg_{g_\e}(x)\\
			&=DH(\nabla_{g_\e}f, \nabla_{g_\e}f)(x)+\f{1}{2}\Div_0\left(|\nabla_{g_\e} f|_{g_\e}^2H\right)(x)-\f{1}{2}|\nabla_{g_\e}f|_{g_\e}^2(x)\Div_0(H)(x),  
		\end{split}
	\end{equation*}
	for all $x\in\Om_\e$. 
\end{lemma}

\begin{proof}
	Let $H=\sum_{i=1}^N H^i\pxi$. Then $H(f)=\sum_{i=1}^N H^i\f{\pt f}{\pt x^i}$, and from \eqref{11.24.3} we have 
	\begin{equation*}
		\begin{split} 
			\nabla_{g_\e}\left(\sum_{i=1}^N H^i\f{\pt f}{\pt x^i}\right)
			&=\sum_{j=1}^N w_\e \pxj\left(\sum_{i=1}^N H^i\f{\pt f}{\pt x^i}\right)\pxj \\
			&=\sum_{i,j=1}^N w_\e \f{\pt H^i}{\pt x^j}\f{\pt f}{\pt x^i}\pxj +\sum_{i,j=1}^N w_\e H^i\f{\pt^2f}{\pt x^j\pt x^i}\pxj. 
		\end{split} 
	\end{equation*}
	Hence
	\begin{equation}\label{11.24.6}
		\begin{split}
			\llg \nabla_{g_\e} f, \nabla_{g_\e}(H(f))\rrg_{g_\e}
			&=\sum_{i,j=1}^N w_\e\f{\pt H^i}{\pt x^j}\f{\pt f}{\pt x^i}\f{\pt f}{\pt x^j}+\sum_{i,j=1}^N w_\e H^i \f{\pt^2 f}{\pt x^i\pt x^j}\f{\pt f}{\pt x^j}
		\end{split}
	\end{equation}
	according to \eqref{11.24.4} and \eqref{11.24.5}. 
	
	On the one hand, we have 
	\begin{equation*}
		\begin{split}
			DH(\nabla_{g_\e}f, \nabla_{g_\e}f)
			&=\llg D_{\nabla_{g_\e}f}H, \nabla_{g_\e}f\rrg_{g_\e} =\sum_{i,j=1}^N w_\e \f{\pt H^i}{\pt x^j}\f{\pt f}{\pt x^i}\f{\pt f}{\pt x^j}+\sum_{i,j,k=1}^Nw_\e H^j\f{\pt f}{\pt x^i}\f{\pt f}{\pt x^k}\Ga_{ij}^k(\e)\\
			&=\sum_{i,j=1}^N w_\e \f{\pt H^i}{\pt x^j}\f{\pt f}{\pt x^i}\f{\pt f}{\pt x^j}-\f{\al}{2}|x|^{\al-2}\sum_{i,j=1}^N H^j x^j\left(\f{\pt f}{\pt x^i}\right)^2
		\end{split}
	\end{equation*}
	by \eqref{11.24.7}, and \eqref{11.24.5}, and \eqref{11.24.9} and \eqref{11.24.8};  
	on the other hand, we have 
	\begin{equation*}
		\begin{split}
			&\f{1}{2}\Div_0\left(|\nabla_{g_\e}f|_{g_\e}^2H\right)-\f{1}{2}|\nabla_{g_\e}f|_{g_\e}^2\Div_0(H)\\
			&=\f{1}{2}\sum_{i=1}^N \left(\pxi \lg \nabla_{g_\e}f, \nabla_{g_\e}f\rg_{g_\e} \right)H^i=\f{1}{2}\sum_{i,j=1}^n \f{\pt w_\e}{\pt x^i}\left(\f{\pt f}{\pt x^j}\right)^2H^i+\sum_{i,j=1}^N w_\e\f{\pt^2f}{\pt x^i\pt x^j}\f{\pt f}{\pt x^j}H^i\\
			&=\f{\al}{2}|x|^{\al-2}\sum_{i,j=1}^N H^ix^i\left(\f{\pt f}{\pt x^j}\right)^2+\sum_{i,j=1}^N w_\e\f{\pt^2f}{\pt x^i\pt x^j}\f{\pt f}{\pt x^j}H^i
		\end{split}
	\end{equation*}
	by \eqref{11.24.4}. 
	Combining these identities with \eqref{11.24.6}, we complete the proof of the lemma. 
\end{proof}

\begin{proposition}\label{11.24.P1}
	Let $y_\e$ be a solution of 
	\begin{equation}\label{11.26.1}
		\pt_{tt}y_\e+\mcA_\e y_\e=0 \mbox{ in }Q_\e.
	\end{equation} 
	
	{\rm(1)} If $H$ is a $C^1$ vector field on $\ol\Om_\e$, then 
	\begin{equation*}
		\begin{split}
			&\iint_{\pt Q_\e}\f{\pt y_\e}{\pt \nu_\e}H(y_\e)\df S\df t+\f{1}{2}\iint_{\pt Q_\e} \left((\pt_ty_\e)^2-|\nabla_{g_\e}y_\e|_{g_\e}^2\right)H\cdot \nu \df S\df t\\
			&=\int_{\Om_\e} [(\pt_ty_\e)H(y_\e)](t)\df x\bigg|_{t=0}^{t=T}\\
			&\hspace{4.5mm}+\iint_{Q_\e} DH(\nabla_{g_\e}y_\e, \nabla_{g_\e}y_\e)\df x\df t+\f{1}{2}\iint_{Q_\e} \left((\pt_ty_\e)^2-|\nabla_{g_\e}y_\e|_{g_\e}^2\right)\Div_0(H)\df x\df t. 
		\end{split}
	\end{equation*}
	Here and in what follows, we denote $\f{\pt y_\e}{\pt \nu_\e}=w_\e \nabla y_\e\cdot \nu$ for each $\e\in (0,\f{1}{8}R_0)$. 
	
	{\rm(2)} If $P\in C^2(\ol\Om_\e)$, then 
	\begin{equation*}
		\begin{split}
			\iint_{Q_\e} P\left((\pt_ty_\e)^2-|\nabla_{g_\e}y_\e|_{g_\e}^2\right)\df x\df t
			&=\int_{\Om_\e} (\pt_ty_\e)y_\e P\df x\bigg|_{t=0}^{t=T}+\f{1}{2}\iint_{Q_\e} y_\e^2 \mcA_\e P\df x\df t\\
			&\hspace{4.5mm}+\f{1}{2}\iint_{\pt Q_\e}y_\e^2 \nabla_{g_\e}P\cdot \nu \df S\df t-\iint_{\pt Q_\e}\f{\pt y_\e}{\pt \nu_\e}y_\e P\df S\df t.
		\end{split}
	\end{equation*}
\end{proposition}

\begin{proof}
	Multiplying \eqref{11.26.1} by $H(y_\e)$, integrating by parts, and using Green's formula, we obtain 
	\begin{equation*}
		\begin{split}
			&\iint_{Q_\e} (\pt_{tt}y_\e)H(y_\e)\df x\df t\\
			&=\iint_{Q_\e} \pt_t\left((\pt_ty_\e)H(y_\e)\right)\df x\df t-\iint_{Q_\e} (\pt_ty_\e)\pt_t[H(y_\e)]\df x\df t\\
			&=\int_{\Om_\e} (\pt_ty_\e(t))[H(y_\e)(t)]\df x\bigg|_{t=0}^T-\f{1}{2}\iint_{Q_\e} H([\pt_ty_\e]^2)\df x\df t\\
			&=\int_{\Om_\e}  [\pt_ty_\e(t)][H(y_\e)(t)]\df x\bigg|_{t=0}^{t=T}+\f{1}{2}\iint_{Q_\e} (\pt_ty_\e)^2\Div_0(H)\df x\df t-\f{1}{2}\iint_{\pt Q_\e} (\pt_ty_\e)^2 H\cdot \nu \df S\df t, 
		\end{split}
	\end{equation*}
	and 
	\begin{equation*}
		\begin{split}
			\iint_{Q_\e} (\mcA_\e y_\e)H(y_\e)\df x\df t
			&=-\iint_{\pt Q_\e}\Div_0[(w_\e\nabla y_\e) H(y_\e)]\df x\df t+\iint_{Q_\e} w_\e \nabla y_\e\cdot \nabla[H(y_\e)]\df x\df t\\ 
			&=-\iint_{\pt Q_\e}\f{\pt y_\e}{\pt \nu_\e}H(y_\e)\df S\df t+\iint_{Q_\e} \lg \nabla_{g_\e}y_\e, \nabla_{g_\e}(H(y_\e))\rg_{g_\e}\df x\df t\\
			&=-\iint_{\pt Q_\e}\f{\pt y_\e}{\pt \nu_\e}H(y_\e)\df S\df t+\iint_{Q_\e} DH(\nabla_{g_\e}y_\e, \nabla_{g_\e}y_\e)\df x\df t\\
			&\hspace{4.5mm}+\f{1}{2}\iint_{Q_\e} \Div_0\left(|\nabla_{g_\e}y_\e|_{g_\e}^2H\right)\df x\df t-\f{1}{2}\iint_{Q_\e} |\nabla_{g_\e} y_\e|_{g_\e}^2\Div_0(H)\df x\df t\\
			&=-\iint_{\pt Q_\e}\f{\pt y_\e}{\pt \nu_\e}H(y_\e)\df S\df t+\iint_{Q_\e} DH(\nabla_{g_\e}y_\e, \nabla_{g_\e}y_\e)\df x\df t\\
			&\hspace{4.5mm}+\f{1}{2}\iint_{\pt Q_\e} |\nabla_{g_\e}y_\e|_{g_\e}^2H\cdot\nu\df S\df t-\f{1}{2}\iint_{Q_\e} |\nabla_{g_\e} y_\e|_{g_\e}^2\Div_0(H)\df x\df t
		\end{split}
	\end{equation*}
	by \eqref{11.24.3}, \eqref{11.24.4}, and Lemma \ref{11.24.L1}, these two equalities imply (1) in this proposition. 
	
	(2) Clearly,
	\begin{equation*}
		\mcA_\e P=-\Div_0(w_\e\nabla P)=-\Div_0(\nabla_{g_\e}P)
	\end{equation*}
	by \eqref{11.24.3}. Note that from \eqref{11.25.7} and \eqref{11.24.4} we have 
	\begin{equation*}
		\begin{split}
			\llg \nabla_{g_\e}y_\e, \nabla_{g_\e}(Py_\e)\rrg_{g_\e}
			&=\llg \nabla y_\e, \nabla (Py_\e)\rrg_0 w_\e =\lg \nabla y_\e, y_\e\nabla P\rg_0 w_\e +\lg \nabla y_\e, P\nabla y_\e\rg_0 w_\e\\
			&=y_\e\lg \nabla_{g_\e}y_\e, \nabla_{g_\e}P\rg_{g_\e}+P|\nabla_{g_\e}y_\e|_{g_\e}^2=P|\nabla_{g_\e}y_\e|_{g_\e}^2+\f{1}{2}\nabla_{g_\e}P(y_\e^2)\\
			&=P|\nabla_{g_\e} y_\e|_{g_\e}^2+\f{1}{2}\Div_0\left(y_\e^2\nabla_{g_\e}P\right)+\f{1}{2}y_\e^2\mcA_\e P, 
		\end{split}
	\end{equation*}
Combining this with the divergence theorem, we obtain 
	\begin{equation*}
		\begin{split}
			\int_{\Om_\e} (\pt_ty_\e)(y_\e P)\df x\bigg|_{t=0}^{t=T}
			&=\iint_{Q_\e}\left[(\pt_{tt}y_\e)(y_\e P)+(\pt_ty_\e)^2P\right]\df x\df t\\
			&=\iint_{Q_\e}  \left[(y_\e P)\Div_0(w_\e\nabla y_\e)+(\pt_ty_\e)^2P\right]\df x\df t\\
			&=\iint_{\pt Q_\e}y_\e P\f{\pt y_\e}{\pt \nu_\e}\df S\df t+\iint_{Q_\e} \left[-\lg \nabla_{g_\e}y_\e, \nabla_{g_\e}(y_\e P)\rg_{g_\e}+(\pt_ty_\e)^2P\right]\df x\df t\\
			&=\iint_{Q_\e}P\left((\pt_ty_\e)^2-|\nabla_{g_\e}y_\e|_{g_\e}^2\right)\df x\df t-\f{1}{2}\iint_{Q_\e}y_\e^2\mcA_\e P\df x\df t\\
			&\hspace{4.5mm}+\iint_{\pt Q_\e}\f{\pt y_\e}{\pt \nu_\e}y_\e P\df S\df t-\f{1}{2}\iint_{\pt Q_\e}y_\e^2\nabla_{g_\e}P\cdot \nu\df S\df t. 
		\end{split}
	\end{equation*} 
	This shows (2) in this proposition. 
\end{proof}

The next lemma identifies the key geometric positivity generated by the multiplier field $H(x)=x$, which plays a crucial role in this paper.

\begin{lemma}\label{11.26.L1}
	Let 
	\begin{equation}\label{11.25.2}
		H=\sum_{k=1}^N H^k\pxk \mbox{ with } H^k=x^k, k=1,\cdots, N.  
	\end{equation}
	Then for all $X\in(\Om_\e)_x$ and $x\in\Om_\e$ we have 
	\begin{equation}\label{11.26.3}
		\begin{split}
			\lg D_X H, X\rg_{g_\e}=a|X|_{g_\e}^2 \mbox{ with } a=\f{1}{2}(2-\al). 
		\end{split}
	\end{equation} 
\end{lemma}

\begin{proof}\label{12.15.1}
	We directly compute 
	\begin{equation*}
		\begin{split}
			\llg D_XH, X\rrg_{g_\e}
			&=\sum_{i,j,k=1}^N X^i X^j\llg D_\pxi H^k \pxk, \pxj\rrg_{g_\e}\\
			&=\sum_{i=1}^N w_\e^{-1}(X^i)^2+\sum_{i,j,k=1}^N w_\e^{-1}X^i X^j H^k \Ga_{ik}^j=\f{1}{2}(2-\al)\sum_{i=1}^N w_\e^{-1}(X^i)^2
		\end{split}
	\end{equation*}
	by \eqref{11.24.9} and $H^k=x^k$ for all $k=1,\cdots, N$. This completes the proof of the lemma. 
\end{proof}

We now introduce some notation that will be used later.

\begin{notation}\label{12.16.N1}
Let $H, a$ be defined in \eqref{11.25.2} and \eqref{11.26.3} respectively. Set 
\begin{equation}\label{11.26.5}
	\begin{split}
		b=\sup_{x\in\Om}|H|_{g_\e}(x),\quad P=\Div_0 H-a.  
	\end{split}
\end{equation}
Note that  for all $x\in \Om_\e$ we have 
\begin{equation*}
	|H|_{g_\e}^2=\sum_{i=1}^N w_\e^{-1}(H^i)^2=|x|^{2-\al}, \mbox{ and } R_0^\f{2-\al}{2}<b\leq M^\f{2-\al}{2}
\end{equation*}
by \eqref{11.25.7}, then 
\begin{equation}\label{11.27.3}
	P=N-a,
\end{equation}  
Denote 
\begin{equation}\label{11.26.7}
	c=N^2-a^2. 
\end{equation} 
We note that $a, b, c$, and $P$ are positive constants and are independent of $\e \in (0, \tfrac{1}{8} R_0)$.
\end{notation}

\begin{lemma}\label{11.26.L2}
	Let $H$ be defined by \eqref{11.25.2}. Let $y_\e$ be a solution of \eqref{12.12.15} with respect to $(y_\e^0, y_\e^1,f_\e=0)$. Then 
	\begin{equation}\label{11.27.1}
		\left|\left( \pt_ty_\e,H(y_\e)+\f{1}{2}Py_\e\right)_{L^2(\Om_\e)}\right|\leq bE_\e(0), 
	\end{equation}
	where 
	\begin{equation}\label{11.27.2}
		E_\e(t)=\f{1}{2}\int_{\Om_\e} \left((\pt_ty_\e)^2+|\nabla_{g_\e}y_\e|_{g_\e}^2\right)\df x=E_\e(0)=\f{1}{2}\left(\|y_\e^1\|_{L^2(\Om_\e)}^2+\left\||\nabla_{g_\e}y_\e^0|_{g_\e}\right\|_{L^2(\Om_\e)}^2\right). 
	\end{equation}
\end{lemma}

\begin{proof}
	Multiplying \eqref{12.12.15} by $\pt_ty_\e$, integrating on $\Om_\e\ts (0,t)$, we get
	\begin{equation*}
		\begin{split}
			0
			&=\iint_{\Om_\e\ts (0,t)}(\pt_{tt}y_\e)(\pt_ty_\e)\df x\df t-\iint_{\Om_\e\ts (0,t)}\left[\Div(w_\e\nabla y_\e)\right]\pt_ty_\e\df x\df t\\
			&=\f{1}{2}\iint_{\Om_\e\ts (0,t)}\pt_t\left(\pt_ty_\e\right)^2\df x\df t+\f{1}{2}\iint_{\Om_\e\ts (0,t)}\pt_t(w_\e\nabla y_\e\cdot \nabla y_\e)\df x\df t
		\end{split}
	\end{equation*}
	by $\pt_ty_\e=0$ on $\pt Q_\e$. This shows that 
	\begin{equation*}
		E_\e(t)=\f{1}{2}\int_{\Om_\e}\left((\pt_ty_\e)^2+|\nabla_{g_\e}y_\e|_{g_\e}^2\right)\df x=E_\e(0)
	\end{equation*}
	for all $t\in [0,T]$. 
	
	From the divergence theorem and $y_\e=0$ on $\pt\Om_\e$ we get
	\begin{equation*}
		\int_{\Om_\e} H(y_\e^2)\df x=\int_{\Om_\e} x\cdot \nabla y_\e^2\df x=-N\int_{\Om_\e} y_\e^2\df x, 
	\end{equation*}
	and hence 
	\begin{equation*}
		\begin{split}
			\left\|H(y_\e)+\f{1}{2}Py_\e\right\|_{L^2(\Om_\e)}^2
			&=\|H(y_\e)\|_{L^2(\Om_\e)}^2+(H(y_\e), Py_\e)_{L^2(\Om_\e)}+\f{1}{4}\|Py_\e\|_{L^2(\Om_\e)}^2\\
			&=\|H(y_\e)\|_{L^2(\Om_\e)}^2+\f{1}{2}\int_{\Om_\e} PH(y_\e^2)\df x+\f{1}{4}\|Py_\e\|_{L^2(\Om_\e)}^2\\
			&=\|H(y_\e)\|_{L^2(\Om_\e)}^2-\f{c}{4}\int_{\Om_\e} y_\e^2\df x
		\end{split}
	\end{equation*}
	by \eqref{11.26.5} and \eqref{11.27.3}, and \eqref{11.26.7}. Therefore, from \eqref{11.25.7} and \eqref{11.26.5} and 
	\begin{equation*}
		\begin{split}
			\|H(y_\e)\|_{L^2(\Om_\e)}^2
			&=\int_{\Om_\e} \left(\sum_{i=1}^N H^i\f{\pt y_\e}{\pt x^i}\right)^2\df x\leq \int_{\Om_\e} \left(\sum_{i=1}^N w_\e^{-1}(H^i)^2\right)\left(\sum_{i=1}^N w_\e \left(\f{\pt y_\e}{\pt x^i}\right)^2\right)\df x\\
			&=\int_{\Om_\e} |H(x)|_{g_\e}^2\left[\sum_{i=1}^N w_\e \left(\f{\pt y_\e}{\pt x^i}\right)^2\right]\df x\leq b^2\left\||\nabla_{g_\e}y_\e|_{g_\e}\right\|_{L^2(\Om_\e)}^2, 
		\end{split}
	\end{equation*} 
	we have 
	\begin{equation*}
		\begin{split}
			\left|\left(\pt_ty_\e, H(y_\e)+\f{1}{2}Py_\e\right)_{L^2(\Om_\e)}\right|
			&\leq \|\pt_ty_\e\|_{L^2(\Om_\e)}\left\|H(y_\e)+\f{1}{2}Py_\e\right\|_{L^2(\Om_\e)}\\
			&\leq \|\pt_ty_\e\|_{L^2(\Om_\e)}\|H(y_\e)\|_{L^2(\Om_\e)}\\
			&\leq b\|\pt_ty_\e\|_{L^2(\Om_\e)}\left\||\nabla_{g_\e}y_\e|_{g_\e}\right\|\leq bE_\e(0).
		\end{split}
	\end{equation*}
	This proves the lemma. 
\end{proof}

\begin{theorem}\label{11.27.T1}
	Let $\al\in (0,1)$. Set 
	\begin{equation}\label{11.27.4}
		\theta=\sup_{x\in\pt\Om_\e}\f{H\cdot\nu}{|\nu_\e|_{g_\e}^2}. 
	\end{equation}
	Then,  for all $T>0$ and all $y_\e^0\in H_0^1(\Om_\e;w_\e), y_\e^1\in L^2(\Om_\e)$, 
	\begin{equation}\label{11.28.8}
		\int_0^T\int_{\Ga_0} \left(\f{\pt y_\e}{\pt \nu_\e}\right)^2\df S\df t\geq \f{2}{\theta}\left(aT-2b\right)E_\e(0), 
	\end{equation}
	where $y_\e$ is the solution of \eqref{12.12.15} with respect to $(y_\e^0,y_\e^1,f_\e=0)$, and $a,b$ are defined in \eqref{11.26.3} and \eqref{11.26.5}, respectively. 
\end{theorem}

\begin{proof}
	Note that from \eqref{11.25.7} and $\nu_\e=w_\e \nu$ we get  $|\nu_\e|_{g_\e}^2(x)=\sum_{i=1}^Nw_\e(\nu^i)^2(x)=|x|^\al$ on $\Om_\e$, where $\nu$ is the outer normal vector on $\pt\Om_\e$. And from  $H\cdot\nu=\sum_{i=1}^N x^i \nu^i$ and $\al\in (0,1)$, we obtain $\f{H\cdot\nu}{|\nu_\e|_{g_\e}^2}\leq |x|^{1-\al}$ on $\pt\Om_\e$, and 
	\begin{equation}\label{11.27.6}
		0\le \theta=\sup_{x\in\pt\Om_\e}\f{H\cdot\nu}{|\nu_\e|_{g_\e}^2}\le M^{1-\al}, \mbox{ and $\theta$ is independent of $\e>0$}. 
	\end{equation}  
	
	Let $x\in \pt\Om_\e$. Then we have 
	\begin{equation}\label{11.28.3}
		\nabla_{g_\e}y_\e(x)=\llg \nabla_{g_\e}y_\e(x), \f{\nu_\e(x)}{|\nu_\e|_{g_\e}}\rrg_{g_\e}\f{\nu_\e(x)}{|\nu_\e|_{g_\e}}+Y_\e(x) \mbox{ with } \lg Y_\e(x), \nu_\e(x)\rg_{g_\e}=0. 
	\end{equation}
	Note that from \eqref{11.23.1} and  \eqref{11.28.1} we get 
	\begin{equation*}
		Y_\e(x)\cdot \nu(x)=\llg Y(x), \nu_\e(x)\rrg_{g_\e}=0, 
	\end{equation*}
	this implies that $Y_\e(x)\in (\pt\Om_\e)_x$, the tangent space of $\pt\Om_\e$ at $x$. Combining this with \eqref{11.28.3} and the boundary condition $y_\e=0$ on $\pt\Om_\e$, we obtain 
	\begin{equation}\label{11.28.4}
		\begin{split}
			\left|\nabla_{g_\e}y_\e\right|_{g_\e}^2
			&=\nabla_{g_\e}y_\e(y_\e)=\f{1}{|\nu_\e(x)|_{g_\e}^2}\llg \nabla_{g_\e}y_\e(x), \nu_\e(x)\rrg_{g_\e}^2+Y_\e(y_\e)=\f{1}{|\nu_\e|_{g_\e}^2}\left(\f{\pt y_\e}{\pt\nu_\e}\right)^2. 
		\end{split}
	\end{equation}
	Similar to \eqref{11.28.3}, $H$ has the form
	\begin{equation}\label{11.28.5}
		\begin{split}
			H=\llg H(x), \f{\nu_\e(x)}{|\nu_\e(x)|_{g_\e}}\rrg_{g_\e}\f{\nu_\e(x)}{|\nu_\e(x)|_{g_\e}}+Z_\e(x), 
		\end{split}
	\end{equation}
	where $Z_\e(x)\in (\pt\Om_\e)_x$. Combining this with \eqref{11.23.1}, \eqref{11.28.1}, and the boundary condition $y_\e=0$ on $\pt\Om_\e$, we obtain 
	\begin{equation}\label{11.28.6}
		\begin{split}
			H(y_\e)=\f{\llg H(x), \nu_\e(x)\rrg_{g_\e}}{|\nu_\e(x)|_{g_\e}^2}\f{\pt y_\e}{\pt \nu_\e}=\f{H(x)\cdot \nu(x)}{|\nu_\e(x)|_{g_\e}^2}\f{\pt y_\e}{\pt \nu_\e}. 
		\end{split}
	\end{equation}
	
	Now, from Assumption \ref{Assumption (H)}, we have
	\begin{equation}\label{11.28.7}
		\begin{split}
			&\f{\theta}{2}\int_0^T\int_{\Ga_0}\left(\f{\pt y_\e}{\pt\nu_\e}\right)^2\df S\df t\\
			&\geq  \f{1}{2}\iint_{\pt Q_\e} \left(\f{\pt y_\e}{\pt \nu_\e}\right)^2\f{H\cdot \nu}{|\nu_\e|_{g_\e}^2}\df S\df t\\
			&=\iint_{\pt Q_\e}\f{\pt y_\e}{\pt\nu_\e}H(y_\e)\df S\df t+\f{1}{2}\iint_{\pt Q_\e}\left((\pt_ty_\e)^2-|\nabla_{g_\e}y_\e|_{g_\e}^2\right)H\cdot \nu\df S\df t\\
			&=\int_{\Om_\e} [(\pt_ty_\e)H(y_\e)](t)\df x\bigg|_{t=0}^{t=T}\\
			&\hspace{4.5mm}+\iint_{Q_\e} \llg D_{\nabla_{g_\e}y_\e } H, \nabla_{g_\e}y_\e\rrg_{g_\e}\df x\df t+\f{1}{2}\iint_{Q_\e} \left((\pt_ty_\e)^2-|\nabla_{g_\e}y_\e|_{g_\e}^2\right)\Div_0(H)\df x\df t\\
			&= \f{a}{2}\iint_{Q_\e} \left((\pt_ty_\e)^2+|\nabla_{g_\e}y_\e|_{g_\e}^2\right)\df x\df t\\
			&\hspace{4.5mm}+\iint_{Q_\e} \llg D_{\nabla_{g_\e}y_\e } H, \nabla_{g_\e}y_\e\rrg_{g_\e}\df x\df t-a\iint_{Q_\e} |\nabla_{g_\e}y_\e|_{g_\e}^2\df x\df t\\
			&\hspace{4.5mm}+\int_{\Om_\e} [(\pt_ty_\e)H(y_\e)](t)\df x\bigg|_{t=0}^{t=T}+\f{1}{2}\iint_{Q_\e} \left((\pt_ty_\e)^2-|\nabla_{g_\e}y_\e|_{g_\e}^2\right)P\df x\df t, 
		\end{split}
	\end{equation}
	where the first inequality follows from Assumption \ref{Assumption (H)} and \eqref{11.27.4}, the first equality follows from \eqref{11.28.4}, \eqref{11.28.6}, and $\pt_ty_\e=0$ on $\pt\Om_\e$, the second equality follows from part (1) of Proposition \ref{11.24.P1}, and the last equality uses $\Div_0(H)=N$ together with \eqref{11.27.3}. Then, from part (2) of Proposition \ref{11.24.P1}, \eqref{11.26.3}, \eqref{11.26.5}, \eqref{11.27.3}, and Lemma \ref{11.26.L2}, we get 
	\begin{equation*}
		\begin{split}
			\f{\theta}{2}\int_0^T\int_{\Ga_0}\left(\f{\pt y_\e}{\pt\nu_\e}\right)^2\df S\df t
			&\geq a TE_\e(0)+\int_\Om (\pt_ty_\e)\left(H(y_\e)+\f{1}{2}Py_\e\right)\df x\bigg|_{t=0}^{t=T}\\
			&\geq (aT-2b)E_\e(0). 
		\end{split}
	\end{equation*}
	From $\pt\Om-B(0,R_0)=\pt\Om_\e-B(0,R_0)$ we get \eqref{11.28.8}. This proves the theorem. 
\end{proof}

\begin{corollary}
	Under the assumptions stated in Theorem \ref{11.27.T1}, then 
	\begin{equation*}
		\int_0^T\int_{\Ga_0}\left(\f{\pt y_\e}{\pt \nu}\right)^2\df S\df t\geq \f{2}{\theta M^{2\al}}(aT-2b)E_\e(0),
	\end{equation*}
	where $a,b$ are defined in \eqref{11.26.3} and \eqref{11.26.5}, respectively. 
	Moreover, when $T>\f{2b}{a}$, we have 
	\begin{equation}\label{12.04.2}
		\int_0^T\int_{\Ga_0}\left(\f{\pt y_\e}{\pt \nu}\right)^2\df S\df t\geq CE_\e(0),
	\end{equation}
	where the constant $C>0$ depends only on $\al, R_0, T$ and $M$. 
\end{corollary}

\begin{proof}
	From $\nu_\e=w_\e \nu$, $\Ga_0\s \pt\Om-B(0,R_0)=\pt\Om_\e-B(0,R_0)$ for $0<\e<\f{1}{8}R_0$, Assumption \ref{Assumption (H)}, and Theorem \ref{11.27.T1}, we obtain the desired conclusion. This completes the proof of the corollary.
\end{proof}

\subsection{Approximation}

We now pass from the approximate observability estimate to the observability of the original degenerate equation with respect to $(y^0,y^1,0)$.

\begin{lemma}\label{12.04.L2}
	Let $a,b$ be defined in \eqref{11.26.3} and \eqref{11.26.5}, respectively. Let $y^0\in C_0^\iy(\Om), y^1\in C_0^\iy(\Om)$, and $f=0$. Then, for $T>\f{2b}{a}$, there exist two positive constants $C_1$ and $C_2$, depending only on $\al, R_0, T, M$, and $\Om-B(0,\f{1}{2}R_0)$, such that 
	\begin{equation}\label{12.04.1}
		\begin{split}
			C_1E(0)\leq \int_0^T\int_{\Ga_0}\left(\f{\pt y}{\pt\nu}\right)^2\df S\df t\leq C_2E(0), 
		\end{split}
	\end{equation}
	where $y$ is the weak solution of \eqref{03.06.m} with respect to $(y^0,y^1,f=0)$, and 
	\begin{equation*}
		E(0)=\f{1}{2}\int_\Om \left((y^1)^2+w\nabla y^0\cdot \nabla y^0\right)\df x. 
	\end{equation*}
\end{lemma}

\begin{proof}
	Let $\e\in (0,\de_0)$ with $\de_0=\min\{\f{1}{8}R_0,\dist (\supp y^0,\pt\Om), \dist (\supp y^1, \pt\Om)\}$.  Then $E(0)=E_\e(0)$ for all $\e\in (0,\de_0)$. And hence the second inequality in \eqref{12.04.1}  follows from   \eqref{12.11.1}. 
	
	From $y^0\in C_0^\iy(\Om)$ we get $y^0\in D(\mcA)\cap D(\mcA_\e)$ for $0<\e<\de_0$. Then, from \eqref{12.04.2}, the identity $E(0)=E_\e(0)$ for all $\e\in (0,\de_0)$, and \eqref{12.13.3}, by letting $\e\ra 0$ along the subsequence in \eqref{12.13.3}, we obtain 
	\begin{equation*}
		\int_0^T\int_{\Ga_0}\left(\f{\pt y}{\pt\nu}\right)^2\df S\geq C_1E(0). 
	\end{equation*}
	This completes the proof of this lemma. 
\end{proof}

\begin{theorem}\label{12.04.T3}
	Let $a,b$ be defined in \eqref{11.26.3} and \eqref{11.26.5}, respectively. Let $y^0\in H_0^1(\Om;w), y^1\in L^2(\Om)$ and $f=0$. Then, for $T>\f{2b}{a}$, there exist two positive constants $C_1$ and $C_2$, depending only on $\al, R_0, T, M$, and $\Om-B(0,\f{1}{2}R_0)$, such that 
	\begin{equation}\label{12.04.3}
		\begin{split}
			C_1E(0)\leq \int_0^T\int_{\Ga_0}\left(\f{\pt y}{\pt\nu}\right)^2\df S\df t\leq C_2E(0),
		\end{split}
	\end{equation}
	where $y$ is the weak solution of \eqref{03.06.m} with respect to $(y^0,y^1,f=0)$. 
\end{theorem}

\begin{proof}
	The second inequality in \eqref{12.04.3} follows from \eqref{12.11.1} in Theorem \ref{12.11.T1}. 
	
	Let $\{\vp_\e^0\}_{0<\e<1}\s C_0^\iy(\Om)$ such that $\vp_\e^0\ra y^0$ in $H_0^1(\Om;w)$, and $\{\vp_\e^1\}_{0<\e<1}\s C_0^\iy(\Om)$ such that $\vp_\e^1\ra y^1$ in $L^2(\Om)$. Consider the following equation
	\begin{equation*}
		\begin{cases}
			\pt_{tt}\vp_\e-\mcA \vp_\e=0, &\mbox{in }Q, \\
			\vp_\e=0, &\mbox{on }\pt Q, \\
			\vp_\e(0)=\vp_\e^0, \pt_t\vp_\e(0)=\vp_\e^1, &\mbox{in }\Om, 
		\end{cases}
	\end{equation*}
	then from \eqref{12.11.1} in Theorem \ref{12.11.T1} we obtain 
	\begin{equation*}
		\left\|\f{\pt (\vp_\e-y)}{\pt \nu}\right\|_{L^2(0,T; L^2(\pt\Om-B(0,R_0)))}\leq C\left(\|y^0-\vp_\e^0\|_{H_0^1(\Om;w)}+\|y^1-\vp_\e^1\|_{L^2(\Om)}\right), 
	\end{equation*}
	i.e., $\f{\pt\vp_\e}{\pt \nu}\ra \f{\pt y}{\pt \nu}$ strongly in $L^2(\pt Q)$ as $\vp_\e^0\ra y^0$ in $H_0^1(\Om;w)$ and $\vp_\e^1\ra y^1$ in $L^2(\Om)$, 
	where the constant $C>0$  depends only on $\al, R_0, T$ and $M$ and $\Om-B(0,\f{1}{2}R_0)$. Note that from $\Ga_0\s \pt\Om-B(0,R_0)$ (see Assumption \ref{Assumption (H)}) we have 
	\begin{equation*}
		\int_0^T\int_{\Ga_0}\left(\f{\pt \vp_\e}{\pt\nu}\right)^2\df S\df t\geq C_1\int_\Om \left((\vp_\e^1)^2+w|\nabla \vp_\e^0|^2\right)\df x
	\end{equation*}
	by Lemma \ref{12.04.L2}, 
	where the constant $C_1>0$ depends only on $\al, R_0, T$ and $M$. Hence
	\begin{equation*}
		\int_0^T\int_{\Ga_0}\left(\f{\pt y}{\pt\nu}\right)^2\df S\df t\geq C_1\int_\Om \left((y^1)^2+w|\nabla y^0|^2\right)\df x=C_1E(0)
	\end{equation*}
	by letting $\e\ra 0$. 
	This completes the proof of this theorem.
\end{proof}

\subsection{Observability}

We say that equation \eqref{03.06.m} is observable at time $T>0$ if the corresponding solution with initial data $(y^0,y^1)$ and $f=0$ satisfies the observability inequality
\begin{equation}\label{12.16.1}
	\begin{split}
		E(0)\leq C\int_0^T\int_{\Ga_0}\left(\f{\pt y}{\pt \nu}\right)^2\df S\df t, 
	\end{split}
\end{equation}
where $C>0$ is a constant, and 
\begin{equation*}
	E(0)=\f{1}{2}\iint_Q \left(w\nabla y^0\cdot \nabla y^0+|y^1|^2\right)\df x\df t.
\end{equation*}

From \eqref{12.16.1}, we can get Theorem \ref{12.04.T4} in the following. 

\begin{proof}[Proof of Theorem \ref{12.04.T4}]
	The observability inequality \eqref{12.16.1} is \eqref{12.04.3}. We complete the proof of this theorem.
\end{proof}

\end{document}